\documentclass[12pt]{amsart}
\usepackage{bookman}
\usepackage[pagebackref]{hyperref}
\usepackage[alphabetic,backrefs,lite,nobysame]{amsrefs} 
\usepackage{amsmath}
\usepackage{url,amssymb,enumerate,colonequals,bm}
\usepackage[headings]{fullpage}
\usepackage[all]{xy}
\usepackage{comment}

\numberwithin{equation}{section}
\numberwithin{figure}{section}

\usepackage{mathrsfs}
\usepackage{tikz-cd}

\usepackage{color}
\usepackage{calrsfs}

\newcommand{\Aut}{\operatorname{Aut}}

\newcommand{\ab}{{\operatorname{ab}}}

\newcommand{\tors}{{\operatorname{tors}}}

\newcommand{\GL}{\operatorname{GL}}

\newcommand{\M}{\operatorname{M}}

\newcommand{\To}{\longrightarrow}

\newcommand{\calF}{{\mathcal F}}

\newcommand{\A}{{\mathcal A}}
\newcommand{\Z}{{\mathbb Z}}

\newcommand{\C}{{\mathbb C}}

\newcommand{\N}{{\mathbb N}}

\newcommand{\calS}{{\mathcal S}}

\newcommand{\qq}{{\mathfrak q}}
\newcommand{\gp}{{\mathfrak p}}

\newcommand{\gm}{{{\mathbf G}_m}}
\newcommand{\ga}{{{\mathbf G}_a}}
\def\st{{\sf t}}
\newcommand{\bL}{\mathbf L}
\newcommand{\bK}{\mathbf K}

\newcommand{\bF}{\mathbf F}
\newcommand{\bx}{\mathbf x}

\newtheorem*{theorem*}{Theorem}
\newtheorem*{conjecture*}{Conjecture}
\newtheorem{conjecture}[equation]{Conjecture}
\newtheorem*{question*}{Question}

\newtheorem{theorem}[equation]{Theorem}
\newtheorem{lemma}[equation]{Lemma}
\newtheorem{corollary}[equation]{Corollary}
\newtheorem{prop}[equation]{Proposition}
\newtheorem{defprop}[equation]{Definition/Proposition}
\newtheorem{proposition}[equation]{Proposition}

\newtheorem{consequence}[equation]{Consequence}
\newtheorem{construction}[equation]{Construction}
\newtheorem*{problem}{Problem}

\theoremstyle{definition}
\newtheorem{definition}[equation]{Definition}
\newtheorem{question}[equation]{Question}
\newtheorem{example}[equation]{Example}
\newtheorem{assumption}[equation]{Assumption}
\newtheorem{remark}[equation]{Remark}
\newtheorem{remarks}[equation]{Remarks}

\theoremstyle{remark}

\newtheorem{notation}[equation]{Notation}

\newcommand{\oo}{\mathcal{O}}
\newcommand{\Q}{\mathbb{Q}}
\newcommand{\R}{\mathbb{R}}
\newcommand{\PP}{\mathbb{P}}

\newcommand{\rank}{\mathrm{rank}}

\newcommand{\Gal}[2]{\mathrm{Gal}(#1/#2)}

\def\Hom{\mathrm{Hom}}
\def\Spec{\mathrm{Spec}}
\def\Aff{\mathrm{Aff}}
\def\df{\bf}

\makeatother

\title[ Existential Definability and Diophantine Stability]{Existential Definability 
and Diophantine Stability}
\author{Barry Mazur}
\address{Department of Mathematics\\ Harvard University \\ Cambridge, MA 02138-2901}
\email{mazur@g.harvard.edu}
\urladdr{http://www.math.harvard.edu/~mazur}
\author{Karl Rubin}
\address{Department of Mathematics\\
UC Irvine\\
Irvine, CA 92697, 
USA}
\email{krubin@uci.edu}
\urladdr{https://math.uci.edu/~krubin}
\author{Alexandra Shlapentokh}
\address{Department of Mathematics \\ East Carolina University \\ Greenville, NC 27858}
\email{shlapentokha@ecu.edu }
\urladdr{myweb.ecu.edu/shlapentokha}


\thanks{
The authors thank Hector Pasten for his comments and suggestions for the paper. They also thank 
Laurent Moret-Bailly for pointing out an error in an earlier version of the paper.\\
The research for this paper was partially supported by DMS grants 2152098 (AS), 2152149 (BM), 2152262 (KR).
AS was also partially supported by an ECU Creative Activity and Research Grant during the summer of 2022.
}

\begin{document}
\maketitle
\setcounter{tocdepth}{1}

\tableofcontents

\section{Introduction}
For the definition of ``diophantine definition'', 
``diophantine undecidable'', ``existentially definable'', 
and similar terms used in this introduction, see 
\S\ref{htpsec} and \S\ref{sec:Dioph}.

Here is a corollary of one of our main results (see Theorems \ref{mainth1} and \ref{thm:inf} below).
By ``number field'' we mean a finite extension of $\Q$.

\begin{theorem}
\label{pre-mainth1}
Let $K$ be a number field, let $L$ be an algebraic (possibly infinite degree) extension of $K$, and let
$\oo_K \subset \oo_L$ be their rings of integers.  Suppose $A$ is an abelian variety 
defined over $K$ such that $A(K)$ is infinite and $A(L)/A(K)$ is a torsion group.  
If at least one of the following conditions is satisfied:
\begin{enumerate}
\item
$L$ is a number field,
\item
$L$ is totally real,
\item
$L$ is a quadratic extension of a totally real field,
\end{enumerate}
then $\oo_K$ has a diophantine definition over $\oo_L$. 
\end{theorem}

Theorem \ref{pre-mainth1} answers a question raised by B.\ Poonen in \cite[Question 2.6(3)]{Po}.
For a discussion of previous results in the direction of Theorem \ref{pre-mainth1}, 
see \S\ref{htpsec}.

Theorem \ref{pre-mainth1} is a special case of a more general result 
(see Propositions \ref{prop2}, \ref{prop3}, and \ref{prop5}) where 
the abelian variety is replaced by a smooth group scheme over $\oo_K$ satisfying 
some mild additional hypotheses (Assumption \ref{Ghyp}).  In this paper the group 
schemes we mainly use are N{\'e}ron models of abelian varieties and twists of 
multiplicative groups.

One way of describing the main structural element in the proof of Theorem \ref{pre-mainth1} is the following. We pass, via an `existential definition'---from an appropriate subgroup of the group of  rational points on  the group scheme we work with---to the (additive group of the ) ring $\oo_K$, by constructing a scheme-theoretic (existentially defined) analogue to some approximation of the standard logarithm mapping that sends an appropriate open subset of the identity in an algebraic group to its Lie algebra. 

The following concept, implicit in the statement of Theorem \ref{pre-mainth1}, is fundamental 
to our results in this paper.

\begin{definition}
\label{def0}
Let $L/K$ be an extension of fields, and $V$ an algebraic variety 
defined over $K$. We denote by $V(K)$ the set of $K$-rational points of $V$.
\begin{itemize} 
\item Say that $V$ is {\df diophantine stable} for $L/K$, or 
$L/K$ is {\df diophantine stable} for $V$,
if the inclusion $V(K) \hookrightarrow V(L)$ is an isomorphism, 
i.e., if $V$ acquires no new rational points after passing from $K$ to $L$.
\item If $V = A$ is an abelian variety over a field $K$, say that $A$ is 
{\df rank stable} for $L/K$ if $A (L)/A(K)$ is a torsion group. If $L$ is a number 
field, this is equivalent to saying that $\rank\, A(K) = \rank\, A(L)$.
\end{itemize}
\end{definition}

A study of diophantine stability for elliptic curves can be found in \cite{MR}, 
and for higher dimension abelian varieties in \cite{MR1}.
  
\begin{definition}\label{chain}
We say that number fields $L/K$ are {\df connected by a diophantine chain} 
if there is an $n \ge 0$ and a tower of number fields  
$$K = K_0 \subset K_1 \subset \cdots \subset K_n \supset L$$
such that for every $i$, $1 \le i \le n$, there is an abelian variety 
$A_i$ defined over $K_{i-1}$ that is rank stable for $K_i/K_{i-1}$ and 
such that $\rank\,A_i(K_{i-1}) > 0$.
\end{definition}

The existence of a rank stable abelian variety {\it descends} in the following sense.
 
\begin{theorem}
\label{descend} 
Let $L/K$ and $K'/K$ be linearly disjoint number field extensions of $K$, and consider 
$L':=K' L$, the compositum of $K'$ and $L$. If there is an abelian variety $A'$ 
over $K'$ with $\rank\, A'(K') = \rank\, A'(L') > 0$, then there is an abelian variety $A$ 
over $K$ with $\rank\, A(K) = \rank\, A(L) > 0$.
\end{theorem}

(See Theorem \ref{mainth2} below.)

\begin{remark}  
The abelian variety $A$ in the conclusion of  Theorem  \ref{descend} 
is the ``$K'/K$-Weil trace'' of $A'$. In particular, $\dim A = [K':K]\dim A'$ and 
$\rank\, A(K) = \rank\, A'(K')$.
\end{remark}  

The following corollary follows directly from Theorem \ref{pre-mainth1} and 
Lemma \ref{3.5} below. 

\begin{corollary}
\label{c2} 
If $L/K$ is an extension of number fields connected by a diophantine chain, 
then the ring of integers of $K$ has a diophantine definition over the ring of 
integers of $L$.
\end{corollary}

\begin{conjecture}
\label{ccc} 
Every number field $L$ is connected to $\Q$ by a diophantine chain.
\end{conjecture} 

A consequence of Conjecture \ref{ccc} is the following 
conjecture, first formulated by Denef and Lipshitz in \cite{Den2}, which is also 
known to follow from other standard conjectures about elliptic curves (\cite{MR,MP18}):

\begin{conjecture}
The ring $\Z$ of rational integers has a diophantine definition over the ring of 
integers of any number field. Hence Hilbert's Tenth Problem has a negative answer 
for the ring of integers of every number field.
\end{conjecture}

For a discussion of Hilbert's Tenth Problem see \S\ref{htpsec}.

Inspired by conjectures of C. David, J. Fearnley, and H. Kisilevsky \cite{DFK1, DFK2},
the first two authors of this article developed in \cite{MR2} a ``heuristic'' 
(based on the statistics of modular symbols)
for groups of rational points on elliptic curves over infinite abelian extensions 
of $\Q$.
Using specially constructed abelian varieties, this heuristic and the main 
results of this article led us to make the following  
diophantine undecidability conjecture.  See \S\ref{stable}, \S\ref{trf}, and \S\ref{qtrf}, especially 
Consequences \ref{newcon}, \ref{con:real7}, and \ref{con:ab7} for details.
Let $\Q^\ab$ denote the maximal abelian extension of $\Q$, the field generated 
over $\Q$ by all roots of unity.

\begin{conjecture}
\label{CON1}
For the primes $p=7$, $11$, or $13$ there are subfields 
${\bL} \subset {\Q}^{\ab}$ for which the field extension 
$ {\Q}^{\ab}/{\bL}$ is cyclic of degree $p$ and such that 
$\oo_{\bL}$, the ring of integers in ${\bL}$, is diophantine 
undecidable. 
\end{conjecture}

In fact, we conjecture (the stronger statement) that ${\Z}$ is existentially 
definable over ${\mathcal O}_{\bL}$ for the fields ${\bL}$ in Conjecture \ref{CON1}. 
We do not make the same conjecture for $\oo_{{\Q}^{\rm ab}}$. 

The following result---related to Conjecture \ref{CON1}---is due to  
the third author \cite{Sh37},  K. Kato \cite{Kato}, K. Ribet \cite{Rib} 
and D. Rohrlich \cite{Rohr1, Rohr2} (see \cite[Theorem 1.2]{LR}).

\begin{theorem}\label{KRR}
Let $\bL$ be an abelian extension with finitely many ramified primes.  
Then $\Z$ is existentially definable over $\oo_{\bL}$.
\end{theorem}

\part{Existential definability and Hilbert's Tenth Problem }

\section{Hilbert's Tenth Problem over rings of algebraic integers}
\label{htpsec}

The original ``Hilbert's Tenth Problem" was one of 23 problems posed over 
a century ago by David Hilbert in the International Congress of Mathematicians, 
at the Sorbonne, in Paris:
\begin{problem} 
Find an algorithm that, when given an arbitrary polynomial equation in several variables over
$\Z$, answers the question of whether that equation has solutions in $\Z$.
\end{problem}

Work of M.\ Davis, H.\ Putnam, J.\ Robinson and Yu.\ Matijasevich shows that 
there is no such algorithm. (See \cite{Da1} and \cite{Da2}.) 

Since the time when this result was obtained, similar questions have been raised for
other fields and rings. E.g., 
\begin{question}
Let $R$ be a computable ring, i.e., a countable ring computable as a set 
and with ring operations represented by computable functions. 
Is there an algorithm (equivalently computer program) taking the coefficients of 
an arbitrary polynomial over $R$ as its input and outputting a ``Yes'' or ``No'' 
answer to the question whether the polynomial in question has solutions in $R$?
\end{question} 

This question in the special case of $R = \Q$ remains an open basic 
diophantine issue; we wonder {(a)} why Hilbert didn't formulate this 
question as an addendum to his initial ``tenth problem," and {(b)} 
whether there is currently a strong consensus guess by the experts about its answer.

One way to resolve the question of diophantine decidability negatively over a ring of
characteristic zero is to construct a diophantine definition of $\Z$ over such a ring. 
The usefulness of such a diophantine definition stems from the fact that if a ring 
has a diophantine definition of $\Z$, then its analog of Hilbert's Tenth Problem is 
undecidable.  We explain how diophantine definitions are used in the following section.

\subsection*{Deriving undecidability of Hilbert's Tenth Problem over a ring 
$R$ using a diophantine definition of $\Z$ over $R$} 

We start with explaining what a diophantine definition is.
\begin{definition}
\label{defdd}
Let $R$ be a ring and let $E$ be a subset of $R$. Then we say that $E$ {\df has a diophantine definition over
$R$} if there exists a finite system polynomials with coefficients in $R$,
$$
\fbox{${{\mathcal F}: \; \text{$f_i(t, x_1,x_2,x_3,\dots x_n) 
\;\in\; R[t,x_1,\ldots,x_n]$ for $i=1,2,\dots, m$}}$}
$$
such that for any $\tau \in R$,
\[ 
\tau \in E \Longleftrightarrow
\text{$\exists a_1,\ldots,a_n \in R$ such that $f_i(\tau,a_1,...,a_n) = 0$ for $i=1,2,\dots, m$} .
\]
We will use interchangeably the terminology $E$ {\em has a diophantine definition over $R$}, 
or $E$ {\em is diophantine over $R$}, or {\em $E$ 
is existentially definable over $R$}---noting that this last is a slight abuse of language: 
properly speaking, we should say $E$ {\em is positively existentially definable over $R$.}
\end{definition}

We now prove an easy proposition that explains the importance of diophantine definition discussed above.

\begin{proposition}
\label{prop:diophdef}
\begin{enumerate}
\item 
Suppose $R$ is a ring containing $\Z$, and $\Z$ has a diophantine definition over $R$.  Then 
there is no algorithm to determine whether an arbitrary finite system of polynomial 
equations with coefficients in $R$ has solutions in $R$.
\item
More generally, suppose $I$ is an arbitrary index set, and $\{R_\alpha : \alpha \in I\}$ 
is a collection of subrings of some ring $\hat{R}$ containing $\Z$. 
Let $R_0 := \cap_{\alpha\in I} R_\alpha$, and suppose there exists a finite collection 
of polynomials 
$$
f_i(t,x_1,\ldots,x_{n}) \in R_0[t,x_1,\ldots x_n], \quad1 \le i \le m
$$ 
that constitutes a  diophantine definition of $\Z$ over $R_{\alpha}$ for {\em every} $\alpha \in I$.  
Then there is no algorithm to determine whether an arbitrary finite system of polynomial 
equations with coefficients in $R_0$ has solutions in $R_{\alpha}$ for {\em some} $\alpha \in I$.
\end{enumerate}
\end{proposition}

\begin{proof}
Assertion (1) is a special case of (2), where we take $I$ to have only one element. 
As for (2), let $p(t_1,\ldots,t_r) \in \Z[t_1, \ldots,t_r]$.  Then for $\alpha \in I$, 
the system equations
\begin{equation}
\label{sys}
p(t_1,\ldots,t_r)=0, \quad
f_i(t_j,x_{j,1},\ldots, x_{j,n})=0, \quad 1 \le i \le m, 1 \le j \le r
\end{equation}
has solutions in $R_{\alpha}$ if and only if the equation $p(t_1,\ldots,t_r)=0$ 
has solutions in $\Z$.  So if there is an algorithm to to determine whether 
\eqref{sys} has solutions in $R_{\alpha}$ for some $\alpha$, then there 
is an algorithm to determine whether $p(t_1,\ldots,t_r)$ has solutions in $\Z$.
\end{proof}

Here is a brief account of some of the history of diophantine definitions.
Using norm equations, diophantine definitions have been obtained for $\Z$ over the 
rings of algebraic integers of some number fields. J. Denef has constructed a 
diophantine definition of $\Z$ for the finite degree totally real extensions of $\Q$. 
J. Denef and L. Lipshitz extended Denef's results to all quadratic extensions of 
finite degree totally real fields. 
(These fields include all finite abelian extensions of $\Q$.)
T. Pheidas, C. Videla and the third author of this paper 
have independently constructed diophantine definitions of $\Z$ for number 
fields with exactly one pair of non-real conjugate
embeddings. 
Lemma \ref{3.5} below shows that the subfields of all the fields mentioned above ``inherit'' 
the diophantine definitions of $\Z$. 
The proofs of the results listed above can be found in 
\cite{Den1}, \cite{Den2}, \cite{Den3}, \cite{Ph1}, \cite{Vid89}, \cite{Sha-Sh}, and \cite{Sh2}. 

The first abelian varieties put to use for the purpose of definability were 
elliptic curves. Perhaps the first mention of elliptic curves in the context 
of the first-order definability belongs to R. Robinson in \cite{RRobinson} 
and in the context of existential definability and diophantine stability relative to $\Q$ to J. Denef in \cite{Den3}. 
Using elliptic curves B. Poonen has shown in \cite{Po} that if for a number 
field extension $M/K$ we have an elliptic curve $E$ defined 
over $K$, of rank one over $K$, such that
the rank of $E$ over $M$ is also one, then $\oo_K$ (the ring of integers of $K$) 
is diophantine over $\oo_M$. G.~ Cornelissen, T.~Pheidas and K. Zahidi weakened 
somewhat the assumptions of B. Poonen's theorem. Instead of requiring a rank one 
curve retaining its rank in the extension, they require existence of a 
rank one elliptic curve over the number field field 
and an abelian variety or a commutative group-scheme of positive rank defined over $\Q$ and diophantine stable relative to $\Q$ (see \cite{CPZ}). This paper was the first to use a higher dimensional abelian 
variety or a group-scheme to show that Hilbert's Tenth Problem is undecidable over a ring of 
integers of a number field.

Somewhat later B. Poonen and the third author have independently shown 
that the conditions of B. Poonen's theorem can be weakened to
remove the assumption that the rank is one and require 
only that the rank in the extension is positive and the same as the rank over 
the ground field, i.e. the elliptic curve is rank stable and with a positive rank
(see \cite{Sh33} and \cite{Po3}). Additional use of diophantine stable 
elliptic curves can be found in \cite{CS}, where G. Cornelissen and the third 
author of this paper used elliptic curves to define a subfield of a number 
field using one universal and existential quantifiers. Recent papers by 
N. Garcia-Fritz and H. Pasten (\cite{GFPast}) and by D. Kundu, A. Lei and F. Sprung (\cite{KLS}) also use diophantine stability of elliptic 
curves to construct diophantine definitions of $\Z$ over new families of 
rings of integers of number fields.

The first two authors showed in \cite{MR} that if the Shafarevich--Tate 
conjecture holds over a number field $K$, then for any prime degree cyclic 
extension $M$ of $K$, there exists an elliptic curve of rank one over $K$, 
keeping its rank over $M$. Combined with B. Poonen's theorem, this 
result shows that the Shafarevich--Tate conjecture implies that Hilbert's Tenth 
Problem is undecidable over the rings of integers of any number field. 
While in \cite{MR}, the case of a general extension was reduced to a cyclic 
extension of prime degree, in fact, it would be enough to show that result 
holds for any quadratic extension of number fields. The proof of this 
fact relies on well-known properties of diophantine definitions; see 
Theorem \ref{AS} below. R. Murty and H. Pasten produced another conjectural 
instance where one could use diophantine stability of elliptic curves in 
finite extensions of number fields to show that $\Z$ has a diophantine 
definition in the rings of integers (\cite{MP18}). The authors relied on 
a different set of conjectures for elliptic curves (automorphic, parity 
and the analytic rank 0 part of the twisted Birch and Swinnerton-Dyer conjecture) 
and the results from \cite{Sh33} and \cite{Po3} for their proof. An accessible 
exposition of the proof can be found in \cite{MF19}.  H. Pasten also showed in \cite{Pas22-1} that existential definability of $\Z$ over rings of integers of number fields follows from a well-known conjecture on elliptic surfaces.

\section{Main theorems: diophantine definitions from diophantine stability}
\label{ressec}
\subsection*{Number field results}

\begin{theorem}
\label{mainth1}
Let $L/K$ be a number field extension with $\oo_L/\oo_K$ the corresponding extension 
of their rings of integers. Let $A$ be an abelian variety defined over $K$ such that 
$\rank\,A(L) = \rank\,A(K) \ge 1$. Then $\oo_K$ 
has a diophantine definition over $\oo_L$. 
\end{theorem}

\begin{remark}\label{uniform} 
Regarding the diophantine definitions provided by  Theorem \ref{mainth1}, fix 
the number field  $K$ and choose an   abelian variety $A$ over $K$  with  $\rank\,A(K) \ge 1$.
\begin{enumerate} 
\item  For any positive number  $d$ there is a single  set of equations ${\mathcal F}(K,A,d)$   
(i.e., equations of the form described in Definition \ref{defdd}), with coefficients 
in $\oo_K$, such  that for any field extension $L/K$ of degree $\le d$ such that 
$\rank\,A(L) = \rank\,A(K)$  that set of equations provides a diophantine definition  
of $\oo_K$ over $\oo_L$.  
\item 
By Theorem \ref{intro:1} below (\cite[Theorem 1.2]{MR1}), if $A$ is a non-CM elliptic curve there are (infinitely) many 
integers $d$ for which there exist infinitely many pairwise linearly disjoint extensions 
$L/K$ with  $[L:K]=d$ and $\rank\,A(L) = \rank\,A(K)$.
\item  
Without any restriction on the degree of $L$,   if $L$ is a totally real field 
or a quadratic extension of a totally real field  there is a single  set of equations 
${\mathcal F}(K,A)$ with coefficients in $\oo_K$ that provides a diophantine definition  
of $\oo_K$ over $\oo_L$ for any field extension $L/K$ such that $\rank\,A(L) = \rank\,A(K)$.
\item 
If we set $K=\Q$ or let $K$ be any number field with a diophantine definition of $\Z$ over 
$\oo_K$ then Proposition \ref{prop:diophdef} applies to each collection of fields described above.
\end{enumerate}\end{remark}

For the proof of Theorem \ref{mainth1}, see Proposition \ref{prop2} and Lemma \ref{serretate} 
below. Theorem \ref{mainth1} can be sharpened to:
 
\begin{theorem}
\label{mainth2}
Let $L/K$ be a number field extension with $\oo_L/\oo_K$ the corresponding 
extension of their rings of integers. Let $K'/K$ be a number field extension 
with $K'$ linearly disjoint from $L$ over $K$.  Put $L':=K'L$.  
Suppose there is an abelian variety $A'$ 
over $K'$ with $\rank\, A'(K') = \rank\, A'(L') > 0$.
Then
\begin{enumerate} 
\item  
there is an abelian variety $A$ over $K$ with $\rank\, A(K) = \rank\, A(L) > 0$, 
\item  
$\oo_K$ has a diophantine definition over $\oo_L$. 
\item If $L$ is a totally real field or a quadratic 
extension of a totally real field, the diophantine definition in (2) depends only on 
$K$ and $A$, and not on $L$. \end{enumerate}
\end{theorem}

\begin{proof} 
Set  $A$ to be the ``Weil trace" of $A'$ with respect to the field extension $K'/K$.
Then $A$ is an abelian variety over $K$ of dimension $[K':K]\dim\,A'$, and 
$$
A(K) = A'(K' \otimes_K K) = A'(K'), \qquad A(L) = A'(K' \otimes_K L) = A'(L'),
$$
the last equality because we assumed $L$ and $K'$ are linearly disjoint over $K$.
This shows that $A$ satisfies (1).

Applying Theorem \ref{mainth1} with the abelian variety $A$ shows 
that $\oo_K$ has a diophantine definition over $\oo_L$, 
and (2) follows from the ``transitivity lemma'' (Lemma \ref{trans}).
Assertion (3)  follows from Remark \ref{uniform}.
\end{proof} 

Combined with Theorem \ref{intro:1} below (\cite[Theorem 1.2]{MR1}), we get the following corollary.

\begin{corollary}
For any number field $K$, there is an integer $N_K$ such that for every prime 
$\ell > N_K$, and every positive integer $n$, there exist infinitely many cyclic 
extensions $L/K$ of degree $\ell^n$ such that $\oo_K$ is diophantine over $\oo_L$.
\end{corollary}

\begin{proof}
Let $E$ be a non-CM elliptic curve defined over $K$ such that $\rank(E(K)) >0$. 
(Such a curve always exists.) Then by Theorem \ref{intro:1}, for all sufficiently large 
primes $\ell$ and all $n$, there are infinitely  
many cyclic extensions $L/K$ of degree $\ell^n$ such that $E(L)=E(K)$. 
Now the corollary follows from Theorem \ref{mainth1}.
\end{proof}

\subsection*{Results for infinite algebraic extensions of $\Q$} 
We will generally use boldface letters (e.g., ${\bL, \bK}$) to denote fields of algebraic 
numbers that are allowed to have infinite degree over ${\Q}$, and normal type (e.g., $L, K$) 
for number fields, i.e., fields of finite degree over $\Q$.

\begin{theorem}
\label{thm:inf}
Let ${\bL}$ be an algebraic extension of $\Q$. Assume that ${\bL}$ 
is totally real or a quadratic extension of a totally real field. 
Let ${\bK}$ be a subfield of $\bL$. Let $\oo_{\bL}/\oo_{\bK}$ be 
the corresponding extension of their rings of integers. Let $A$ be an abelian variety 
defined over ${\bK}$ such that $A({\bK})$ contains an element of 
infinite order and $A({\bL})/A({\bK})$ is a torsion group. If ${\bK}$ 
is a number field, then $\oo_{\bK}$ has a diophantine definition over 
$\oo_{\bL}$. If ${\bK}$ is an infinite extension of $\Q$, then 
$\oo_{\bL}$ contains a subset $D$ such that $D$ is diophantine over 
$\oo_{\bL}$ and $\Z \subset D \subset \oo_{\bK}$.
\end{theorem}

The proof of Theorem \ref{thm:inf} is similar to that of Theorem \ref{mainth1}, 
using  Propositions \ref{prop3} and \ref{prop5} below.

\begin{remark}
If $K$ is a number field, then the diophantine definition $f(t,x_1,\ldots,x_{\ell})$ 
of $\oo_K$  over $\oo_L$ or $\oo_{\bL}$ constructed in the proofs of 
Theorems \ref{mainth1}, \ref{mainth2} and \ref{thm:inf} has the property that 
for all $t \in \oo_K$ there exist $x_1, \ldots, x_{\ell} \in \oo_K$ such that 
$f(t,x_1,\ldots,x_{\ell})=0$.  This follows from the fact that we use points of 
$A(K)$ to generate rational integers and then a basis of $K/\Q$ to generate all 
elements of $\oo_K$.
\end{remark}

\begin{corollary}
Suppose $\bK \subset \bL$ and $A$ are as in Theorem \ref{thm:inf}. If $\bK$ is 
a number field then the existential theory of ${\oo_\bL}$ is undecidable. 
Alternatively, Hilbert's Tenth Problem is undecidable over ${\oo_\bL}$.
\end{corollary}

\begin{proof}
By the result of Denef and Lipshitz \cite{Den2} mentioned in \S\ref{htpsec}, $\Z$ has a 
diophantine definition over $\oo_\bK$. Combining this with Theorem \ref{thm:inf} 
and Lemma \ref{3.5} proves the corollary.
\end{proof}

\section{Existential definitions}
\label{sec:Dioph}
\subsection*{The basics of existential definability} 
Recall Definition \ref{defdd} above.


\begin{lemma}
\label{3.1} 
Let $\oo$ be an integral domain whose fraction field $K$ is not algebraically closed, 
and suppose $p(t) \in \oo[t]$ is a (non-constant) polynomial with no root in $K$.
Let 
$$
\{f_i(x_1,x_2,\dots,x_m)=0 : 1 \le i \le r\}
$$ 
be a system of polynomial equations over $\oo$. 
Then there exists a single effectively computable polynomial 
$F(x_1,x_2,\dots,x_m) \in \oo[x_1,x_2,\dots,x_m]$ such that the solutions 
to $F = 0$ in $\oo^m$ are the same as the common solutions in $\oo^m$ 
of the system ${\mathcal F}$.
\end{lemma}

\begin{proof}  The proof is taken from \cite{Sh34}. 
Write $p(t) = a_n t^n+a_{n-1}t^{n-1} + \cdots +a_0$ with $a_i \in \oo$. 
If $f, g \in \oo[x_1,x_2,\dots,x_m]$ and $\bx \in \oo^m$, then 
$$
(a_n f^n+a_{n-1}f^{n-1}g+\cdots+a_1fg^{n-1} + a_0g^n)(\bx)=0 
   \iff f(\bx)=g(\bx)=0.
$$ 
Now we proceed by induction to combine any finite number of polynomials 
$\{f_1,\ldots,f_r\}$ into one.
\end{proof}

Note that although the degree of $F(x_1,x_2,\dots,x_m)$ in Lemma \ref{3.1} 
may be significantly higher than the degree of the polynomials 
$f_i(x_1,x_2,\dots,x_m)$ that comprise the system, the number of variables 
$m$ remains the same.

Here are some easy (and well-known) properties of 
{\em existential definability} we use in this paper. The proofs of many of 
the statements below can be found (among other places) in Chapter 2 of \cite{Sh34}.

\begin{lemma}
The set $\oo^\times$ of units in any commutative ring $\oo$ (with unit) 
is existentially definable over the ring $\oo$. 
\end{lemma}

\begin{proof}
The polynomial $f(t, s):= ts-1 \in \oo[t,s]$ has a zero for 
$t=\alpha \in \oo$ if and only if $\alpha \in \oo^\times$.
\end{proof}

Below, if $\bK$ is a subfield of $\bar{\Q}$, then $\oo_\bK$ will be its ring of integers.
The following lemma is proved in \cite{Sha-Sh}.

\begin{lemma}\label{trans}[Transitivity descent for diophantine definitions]
\label{3.5}
Let ${\bK} \subset {\bL} \subset {\bf H}$ be algebraic possibly infinite extensions of $\Q$.
\begin{itemize}
\item  
If $\oo_{\bK}$ has a diophantine definition over $\oo_{\bL}$, and $\oo_{\bL}$ 
has a diophantine definition over $\oo_{\bf H}$, 
then $\oo_{\bK}$ has a diophantine definition over $\oo_{\bf H}$.
\item 
If ${\bf H}/{\bL}$ is a finite extension and  $\oo_{\bK}$ has a 
diophantine definition over $\oo_{\bf H}$, then $\oo_{\bK}$ has a 
diophantine definition over $\oo_{\bL}$.
\end{itemize}
\end{lemma}

\begin{remark}
\label{red} 
In particular, if ${\bf H}/{\bK}$ is a finite extension and $\oo_{\bK}$ has a diophantine definition over $\oo_{\bf H}$, then $\oo_{\bK}$ has 
a diophantine definition over $\oo_{\bL}$ for every intermediate field $\bL$ such that ${\bK} \subset {\bL} \subset {\bf H}$.
\end{remark}
The following lemma is clear.

\begin{lemma}[Intersection]
\label{cap} 
Let $E_1, E_2 \subset \oo_K$ be subsets each existentially 
definable in $\oo_K$. Then $E_1\cap E_2 \subset \oo_K$ is existentially definable. 
\end{lemma}

The following is due to J. Denef (\cite{Den3}).

\begin{lemma}
\label{3.3} 
Let ${\bK}$ be any field of algebraic numbers. The set $\oo_{\bK}\setminus\{0\}$ 
of non-zero elements of $\oo_{\bK}$ is existentially definable over $\oo_{\bK}$.
\end{lemma}

\begin{proof}
Let $x \in\oo_{\bK}$. We claim that $x \ne 0$ is equivalent to the existential statement
$$
\exists y, z, w \in\oo_K : (2y - 1)(3z - 1) = xw.
$$
For if $x = 0$, then either $y=1/2$ or $z=1/3$, so either $y$ or $z$ is not in 
$\oo_K$. Suppose now $x \ne 0$. Working in the 
number field $K_0:=\Q(x)$, we can factor the principal ideal $(x) = {\mathfrak a} {\mathfrak b}$, 
where $({\mathfrak a},(2)) = 1$, and $({\mathfrak b},(3)) = 1$. (It is possible
that either ${\mathfrak a}$ or ${\mathfrak b}$ is the unit ideal.) 
Choose
$y, z \in\oo_{K_0}$ such that $2y \equiv 1 \pmod{\mathfrak a}$ and 
$3z \equiv 1 \pmod{\mathfrak{b}}.$ Then $(2y-1)(3z-1) \equiv 0 \pmod{\mathfrak{ab}},$
and so $x$ divides $(2y- 1)(3z- 1)$ in $\oo_{K_0}$, and therefore in $\oo_{\bK}$ as well.
\end{proof}

\begin{lemma}
\label{3.7}
Let ${\bf L}/K$ be an algebraic extension possibly of infinite degree, where $K$ is a number field. 
Suppose that there exists a subset $S$ of $\oo_K$ containing $\N$ such that $S$ has 
a diophantine definition over $\oo_{\bL}$. Then $\oo_K$ has a diophantine definition over $\oo_{\bL}$. 
\end{lemma}

\begin{proof}
Let $\alpha \in \oo_{\bL}$ be any element such that $K=\Q(\alpha)$ and consider the 
following subset $E$ of $\oo_{\bL}$:
\[
E=\{x \in \oo_{\bL} : bx=\sum_{i=0}^{[\Q(\alpha):\Q]}a_i\alpha^{i},\; 
\text{with}\; b \not =0,\pm a_i, b \in S\}
\]
Now if $y \in E$, then $y \in K\cap \oo_\bL=\oo_K$ because $b \in\oo_K$ and all 
$a_i \in \oo_K$. Conversely, if $y \in \oo_K$, then $y \in E$, since every 
element of $\oo_K$ can be represented as the sum in the definition of $E$ 
with $b \not=0, a_i \in \Z$. Finally the condition $b \ne 0$ is diophantine 
over $\oo_L$ by Lemma \ref{3.3}.
\end{proof}

The following theorem due to the third author was mentioned in the discussion 
at the end of Section \ref{htpsec}.

\begin{theorem}
\label{AS}
Suppose that for every quadratic extension of number fields $L/K$ we have that 
$\oo_K$ has a diophantine definition over $\oo_L$. Then $\Z$ has a diophantine 
definition over the ring of integers of any number field.
\end{theorem}

\begin{proof}
Let $M$ be a number field. By Remark \ref{red}, without loss of generality 
we can assume that $M$ is Galois over $\Q$.
For any complex embedding $\M \hookrightarrow \C$ consider the corresponding complex 
conjugation which gives an involution $\sigma:M\to M$ of the field $M$. 
Let $M^\sigma\subset M$ be the fixed field of this involution. Since $M/M^\sigma$ is a quadratic 
extension, we obtain---from the assumption in the statement of the theorem---that 
its ring of integers has a diophantine definition in $\oo_M$. By Corollary \ref{cap} 
the same is true for the ring of integers in the intersection 
$$
M^+:=\cap_\sigma{M^\sigma}.
$$ 
That is, $\oo_{M^+}$ has a diophantine definition in $\oo_M$.
Since $M^+$ is totally real, the result of Denef \cite{Den3} discussed above gives us that 
$\Z$ has a diophantine definition in $\oo_{M^+}$. By ``transitivity'' of diophantine 
definitions (Lemma \ref{3.5}) we have that $\Z$ has a diophantine definition over $\oo_M$.
\end{proof}

\subsection*{Total Positivity; replacing inequalities by equations}

\begin{proposition}\label{tpprop}
Let $\bF$ be an algebraic (possibly infinite) extension of $\Q$.  
Let $x,z \in \oo_{\bF}$ with $x\ne z$. Then there exists 
$y_1,\ldots, y_5 \in \oo_{\bF}$ with $y_5 \ne 0$ such that
\begin{equation}\label{eqn1} 
y_5^2(x-z)=y_1^2+y_2^2+y_3^2+y_4^2
\end{equation} 
if and only if for every  embedding $\sigma: \bF \hookrightarrow \R$ 
we have that $\sigma(x) > \sigma(z)$.
\end{proposition}

\begin{proof}  The existence  of $y_1,\ldots, y_5 \in \oo_{\bF}$ with $y_5 \ne 0$ 
implies the inequality  $\sigma(x) > \sigma(z)$ for all real embeddings 
$\sigma: \bF \hookrightarrow \R$. 
  
To go the other way, assume  the inequality  $\sigma(x) > \sigma(z)$ for all real embeddings.
It follows that \eqref{eqn1} has a solution in all real completions of ${\bF}$, 
and hence in all archimedean completions.   
In any non-archimedean completion a quadratic form of dimension four  represents 
every element; so \eqref{eqn1}  has a solution in every completion of ${\bF}$. 
By the   Hasse-Minkowski Theorem (\cite{Shim10} Corollary 27.5)  it has a solution in ${\bF}$.
\end{proof}

\part{On the geometry of group schemes over rings of integers in number fields}

\section{The conormal bundle to a section of a smooth scheme}

Let $S= {\Spec}(\oo)$ where $\oo$ is a Dedekind domain, and let $X\to S$ be a 
morphism of finite type and smooth of dimension $d$. Let $e: S \hookrightarrow X$ 
be a section. We'll refer to the pair $(X,e)$ as an {\df $S$-pointed scheme}. 
Let $I = I_e$ be the sheaf of ideals on $X$ that cut out the section $e$. 
For a general reference to this, see \cite[\S{I.4}]{EGA}, especially Proposition 4.1.2. 

Denote by $X_{e,[2]} = X_{[2]}\subset X$ the subscheme cut out by $I^2$. 
In the language of \cite{EGA}, $X_{e,[2]}$ is a ``formal scheme" with support 
equal to the closed subscheme $e: S \hookrightarrow X_{e,[2]}$ and has $I$, 
restricted to $\oo_{X_{e, [2]}}$, as its ideal of definition (which is an 
ideal of square zero in $\oo_{X_{[2]}}$). 

The pullback $e^*(I/I^2)$ to $S$ is a locally free coherent sheaf of rank $d$ over $S$ 
(the {\df conormal bundle} to the section $e$; for another general reference, cf. \cite{Har}).

Let 
$$
{\mathcal N}= {\mathcal N}_{X,e}:= H^0(S,e^*(I/I^2)) \quad \subset \quad 
{\mathcal R}_{X} = {\mathcal R}_{X,e}:= H^0(S,e^*(\oo_X/I^2)).
$$ 
So ${\mathcal N}_{X,e}$ is the (locally free, rank $d$) $\oo$-module of 
sections of the coherent sheaf $e^*(I/I^2)$ over $S$, viewed as an ideal 
in ${\mathcal R}_{X,e}$, the $\oo$-algebra of global sections (over $S$) of $e^*(\oo_X/I^2)$.

We can write 
$$
{\mathcal R}_{X,e} = \oo \oplus {\mathcal N}_{X,e} = \oo[{\mathcal N}_{X,e}].
$$
where the object on the right is the $\oo$-algebra generated by the $\oo$-module 
${\mathcal N}_{X,e} $ where the square of ${\mathcal N}_{X,e}\subset {\mathcal R}_{X,e}$ 
is zero. (Compare: \cite[Proposition 10.8.11]{EGA}.)

\begin{proposition}[Functoriality]
\label{functor} 
\begin{enumerate}\item There is a canonical isomorphism 
$$
X_{e,[2]} \simeq {\Spec}({\mathcal R}_{X,e}).
$$
\item 
Let $f:(X,e) \to (X',e')$ be a morphism of smooth ($S$-pointed) schemes over $S$. 
Then $f$ induces (via the canonical mapping $I_X\to f^*I_{X'}$) functorial morphisms
$$
\xymatrix@C=10pt{X\ar[rr]^f && X'\\
X_{e,[2]}\ar[rr]^{f_{[2]}}\ar@{^(->}[u] && X'_{e',[2]}\ar@{^(->}[u] \\ 
& S\ar[ul]^e \ar[ur]_{e'}}
$$
and (correspondingly) contravariant functorial $\oo$-homomorphisms
\begin{equation}
\label{contr} 
\raisebox{22pt}{\xymatrix{ {\mathcal N}_{X'}\ar[r]\ar@{^(->}[d] & {\mathcal N}_X\ar@{^(->}[d] \\ 
{\mathcal R}_{X'}\ar[r] & {\mathcal R}_X.}}
\end{equation}
\item 
If $f: X\to X'$ is a closed immersion, so is 
$$
f_{[2]}: X_{[2]}\to X'_{[2]},
$$ 
and the horizontal morphisms in \eqref{contr} are surjections.
\item 
If $(X,e)$ and $(X',e')$ are smooth ($S$-pointed) schemes over $S$, 
letting $(Y,y):=(X,e)\times_S(X',e')$ we have an isomorphism of $\oo$-modules
$$ 
{\mathcal N}_{(Y,y)} \simeq {\mathcal N}_{(X,e)}\oplus {\mathcal N}_{(X',e')}.
$$
\end{enumerate} 
\end{proposition}

\begin{proof} The statements follow directly from the functoriality of the 
construction $X\mapsto X_{[2]}$ and the fact that $e'=f\circ e$ so $(e')^*= e^*\circ f^*$. 
\end{proof}

\section{The conormal bundle to the identity section of a smooth group scheme}
\label{ibis}

We will be dealing with smooth group schemes $G$ of finite type over our base $S$, 
which we now suppose to be ${\Spec}(\oo_K)$ for some number field $K$. 
Our main applications will use group schemes $G$ that are either 
\begin{itemize} 
\item 
the N{\'e}ron model over the base $S$ of an abelian variety $A_{/K}$, or
\item 
the multiplicative group ${\mathbf G}_m$ over $S$, or more generally a torus over $S$, or
\item 
(possibly in the future) extensions of these groups.
\end{itemize}

As the reader will see, we will only be ``using" the connected component of the 
identity of $G$, so we could restrict to connected group schemes over $S$. 
Moreover, there are few properties of $G$ (besides smoothness along the 
identity section) that are required, in the constructions to follow. 
Specifically, $G$ needn't be commutative; it needn't even have inverses: it could 
just be a monoid; more curious is that---although it would take some discussion 
which we won't enter to explain this: it needn't even be associative. The main 
requirement is that there be a binary law $\gamma: G\times G \to G$ of schemes 
over $S$ with a two-sided identity section $e: S \to G$; i.e., such that this 
diagram is commutative:
$$
\xymatrix@C=30pt{G= S\times_S G \ar[r]^-{e\times Id}\ar[rd] 
& G\times_SG\ar[d]^\gamma & G\times_SS=G\ar[l]_-{Id\times e}\ar[ld]\\
& G}
$$
But let $(X,e) = (G,e)$ just be a smooth group scheme of finite type over $S$ 
pointed by its ``identity section." Below we'll begin to drop the $e$ from 
$(G,e)$ and just call it $G$. 

\begin{lemma}
\label{free1} 
Let $h$ denote the class number of the number field $K$. 
For $G$ a smooth group scheme of finite type over $S$, let 
$$
G' :=\{ G\}^h:= G\times_S G\times_S \dots \times_S G
$$ 
denote the $h$-fold power of $G$ , and ${\mathcal N}_{G}$ and ${\mathcal N}_{G'}$ 
their corresponding conormal bundle $\oo_K$-modules. Then
$$
{\mathcal N}_{G'} = \oplus^h {\mathcal N}_{G},
$$ 
and ${\mathcal N}_{G'}$ is a {\em free} (finite rank) $\oo_K$-module.
\end{lemma}

\begin{proof} 
The first part of that sentence follows from Proposition \ref{functor}(4). 
The second follows from the fact that ${\mathcal N}_{G}$ is locally free 
over $\oo_K$, and if $h$ is the class number of $K$ the $h$-fold direct sum 
of any locally free $\oo_K$-module is free (\cite[Theorem II.4.13]{FT}). 
\end{proof}

\begin{remark}
\label{rmk1} 
This lemma will be useful later. Whenever we have a group scheme that is 
diophantine stable for a field extension $L/K$, Lemma \ref{free1} allows 
us to choose one with the further property that its conormal bundle module 
${\mathcal N}_G$ is free over $\oo_K$. 
\end{remark}

\begin{proposition}
\begin{enumerate}
\item 
The functor $(G,e) \mapsto G_{e,[2]}$ preserves closed immersions.
\item 
The functor $G \mapsto {\mathcal N}_G$ sends closed immersions $G_1 \hookrightarrow G_2$ 
to surjections 
$$
{\mathcal N}_{G_2} \To {\mathcal N}_{G_1}.
$$ 
\item 
If $G, H$ are $S$-group schemes we have canonical closed immersions of
$S$-schemes 
$$
\xymatrix{(G\times_SH)_{[2]}\ar[d]^=\ar@{^(->}[r] 
& G_{[2]}\times_SH_{[2]}\ar[d]^=\ar@{^(->}[r] & G\times_SH\\ 
{ \rm Spec}({\mathcal R}_{G\times_S H})\ar@{^(->}[r] 
& {\Spec}({\mathcal R}_{G})\times_S{\rm Spec}({\mathcal R}_{H})& }
$$ 
\item 
Letting $1_G\in {\mathcal R}_G={\mathcal R}_{G,e}$ denote the unit, and ditto for $H$, 
we have a canonical isomorphism of ${\oo_K}$-modules 
$$ 
{\mathcal N}_{G\times_SH} \stackrel{\simeq}{\longrightarrow} 
{\mathcal N}_G\otimes_{\oo_K} 1_H \oplus 1_G\otimes_{\oo_K} {\mathcal N}_H.
$$ 
\item 
Let 
$
\gamma: G \times_SG \rightarrow G
$ 
denote the group law $(g_1,g_2) \mapsto g_1g_2$. We have a commutative diagram
$$
\xymatrix{ {\Spec}({\mathcal R}_{G\times_S G})\ar[r]^=\ar[d]^{\gamma_{[2]}} 
&\{G\times_S G\}_{[2]}\ar@{^(->}[r]\ar[d]^{\gamma_{[2]}} 
& G_{[2]}\times_S G_{[2]}\ar[d]^\gamma\ar@{^(->}[r] & G\times_S G\ar[d]^\gamma\\
{\Spec}({\mathcal R}_{G})\ar[r]^= & G_{[2]}\ar@{^(->}[r] & G\ar[r]^= & G.}
$$
\end{enumerate} 
\end{proposition}

\begin{proof} Items (1) and (2) follow from item (3) of Proposition \ref{functor}. 
The remaining items follow straight from the definitions or the functoriality 
of the objects named, except for (4) which is a direct computation.

Note that the natural $\oo_K$-homomorphism 
$$ 
{\mathcal R}_{G}\otimes_{\oo_K}{\mathcal R}_{H} \longrightarrow {\mathcal R}_{G\times_S H}
$$ 
is a surjection, but not (necessarily) an isomorphism. 
\end{proof} 

\begin{proposition}
\label{propcomm} 
Recall that $\gamma : G \times G \to G$ denotes the group operation. The mapping 
$$
\gamma_{[2]}: {\mathcal N}_G \longrightarrow {\mathcal N}_{G\times_SG} 
= ({\mathcal N}_G\otimes_\oo 1_G) \oplus (1_G\otimes_\oo {\mathcal N}_G)
$$
is given by the formula
\begin{equation}
\label{prod1} 
x \quad\mapsto\quad x\otimes 1_G + 1_G \otimes x.\end{equation}
\end{proposition} 

\begin{proof} 
Since $e$ is the identity section we have the commutative diagram 
$$
\xymatrix{ e\times G\ar@{^(->}[r]\ar[rd]_= & G\times G\ar[d]^\gamma 
& G\times e\ar@{_(->}[l]\ar[ld]^=\\
& G & }
$$ 
which gives us that the composition of 
$$
{\mathcal N}_G \stackrel{\gamma_{[2]}}{\longrightarrow} {\mathcal N}_{G\times_SG} 
= {\mathcal N}_G\otimes_\oo 1_G \oplus 1_G\otimes_\oo {\mathcal N}_G
$$ 
with projection to $ {\mathcal N}_G\otimes_\oo 1_G$ or to $1_G\otimes_\oo {\mathcal N}_G$ 
induces the `identity mapping' (i.e., $x \mapsto x\otimes 1_G$ or 
$x \mapsto 1_G \otimes x$ respectively). 
\end{proof}

\begin{corollary}
Let $G$ and $G'$ be two smooth group schemes over $S$. 
Let $G_0 \subset G$ be an open subscheme containing $e$, the identity section, 
and let $G_0' \subset G'$ be, similarly, an open subscheme containing $e'$, 
the identity section. We view $(G_0,e)$ and $(G_0',e')$ as (smooth) 
$S$-pointed schemes.
Let $\iota: (G_0,e) \to (G_0',e')$ be a morphism of $S$-pointed schemes 
that is a closed immersion of schemes
(but $\iota$ is not required to extend to a homomorphism, or even a morphism, 
of the ambient groups). We have a commutative diagram 
\begin{equation}
\label{comm}
\raisebox{25pt}{\xymatrix@C=15pt{ G\times G\ar[d]^\gamma 
& (G\times G)_{[2]}\ar[l]\ar[r]^=\ar[d]^{\gamma_{[2]}} 
& (G_0\times G_0)_{[2]}\ar[r]^{\iota\times \iota}\ar[d]^{\gamma_{[2]}} 
& (G'_0\times G'_0)_{[2]}\ar[r]^=\ar[d]^{{\gamma'}_{[2]}} 
& (G'\times G')_{[2]}\ar[d]^{\gamma'_{[2]}}\ar@{^(->}[r]& {G'\times G'}\ar[d]^{\gamma'}\\
G & G_{[2]}\ar[l]\ar[r]^= & {G_0}_{[2]}\ar[r]^{\iota} 
& {G_0'}_{[2]}\ar[r]^= & G'_{[2]}\ar@{^(->}[r] & G'.}}
\end{equation} 
\end{corollary} 

\begin{remark} We will make use of this discussion in the case where $G=\A$ 
is the N{\'e}ron model of an abelian variety, and $G'= {\Aff}^n$ is $n$-dimensional 
affine space. The key fact we use is Proposition \ref{propcomm} and the 
commutativity of the inner square in diagram \eqref{comm} (and this follows 
directly from the formula \eqref{prod1}). 
\end{remark}

\section{A prepared group scheme}

Fix a number field $K$. 
\subsection*{ Projective space over $\oo_K$}
For $n$ a positive integer, consider $n$-dimensional projective space $\PP^n$ 
viewed as a scheme (over $\Z$, or more relevant to our context, over $\oo_K$).
A point in $\PP^n$ rational over $K$ (which is the same as being rational 
over $\oo_K$) can be represented in $n+1$ homogeneous coordinates (not all 
of them $0$), $$(x_1: x_2:\dots: x_{n+1})$$ for $x_i \in K$ noting that such 
a representation is unique up to scalar multiplication by a nonzero element in $K$. 
Two such vectors 
$(x'_1: x'_2:\dots: x'_{n+1}), (x_1: x_2:\dots: x_{n+1})$ are {\em equivalent} 
if and only if there is a nonzero element $c \in K$ such that $x_i'=c x_i$ for $i=1,2,\dots n+1.$
Any such point can therefore be represented by such a vector with $x_i \in \oo_K$. 
(Below, we keep to the convention that the colons signify that we are 
considering ``homogenous coordinates.'') 

\begin{definition}
In the special case that the entries $a_1,a_2,\dots, a_{n+1}$ of an $n+1$-vector
$$
\alpha= (a_1: a_2:\dots: a_{n+1})
$$ 
generates a principal ideal $``\gcd(\alpha)" = (a) \subset \oo_K,$ define 
\begin{itemize} 
\item 
the {\df denominator} of $\alpha$ to be 
$$
\delta(\alpha):= a_{n+1}/a \subset \oo_K,
$$
noting that $\delta(\alpha)$ is only well-defined up to a unit in $K$---i.e., 
it is only the principal ideal generated by $\delta(\alpha)$ that is well-defined,
\item 
the {\df numerator} of $\alpha$ to be 
$$
\nu(\alpha):= \text{the ideal $(a_1/a,a_2/a,\dots, a_n/a) \subset \oo_K,$}
$$
\end{itemize}
noting that $\nu(\alpha)$ and $\delta(\alpha)$ are relatively prime---i.e., 
the ideal generated by $\nu(\alpha)$ and $\delta(\alpha)$ is the unit ideal.
\end{definition}

\subsection*{ Affine space over $\oo_K$}
\label{aff}

We view affine $n$-dimensional space 
$$
{\Spec}(\oo_K[y_1,y_2,\dots, y_n]) =:{\Aff}^n \simeq {\mathbf G}_a^n
$$ 
as an additive group scheme over $\oo_K$; with 
$$
e:=(0,0,\dots,0) \in {\Aff}^n
$$ 
its zero-section (cut out by the ideal 
$I:=(y_1, y_2,\dots, y_n) \subset \oo_K[y_1, y_2,\dots, y_n]$).

Letting $\PP^{n-1} \simeq H \subset \PP^n$ denote the hyperplane defined 
by $x_{n+1}=0$, we have an isomorphism
$$
{\Aff}^n \simeq \PP^n\setminus H \subset \PP^n,
$$ 
defined by 
$$
(a_1,a_2,\dots,a_n) \mapsto (a_1:a_2:\dots:a_n:1).
$$

``Going the other way:'' if $(a_1:a_2:\dots:a_n: a_{n+1})$ is a homogenous 
representative of a point ${\mathbf a} \in \PP^n(\oo_K)$, where 
$a_{n+1} \ne 0$, denote by 
\begin{equation}
\label{pa} 
{\mathbf a}^!:=({\frac{a_1}{a_{n+1}}}, {\frac{a_2}{a_{n+1}}},\dots,{\frac{a_n}{a_{n+1}}}) 
\in {\Aff}^n\big(\oo_K[{\frac{1}{a_{n+1}}}]\big).
\end{equation}

\begin{remark}
\label{frac} 
If the ideal $(a_1,a_2,\dots, a_n) \in \oo_K$ is 
relatively prime to the ideal $(a_{n+1})$ then ${\mathbf a}^!$ reduces to 
a well-defined element---call it ${\mathbf a}^!_{[2]}$---in the quotient 
$\oo_K^n/(a_1,a_2,\dots, a_n)^2\oo_K^n$. 
\end{remark}

Visibly, ${\mathbf a}^!_{[2]}$ is dependent only on the equivalence class of 
$(a_1:a_2:\dots:a_n: a_{n+1})$ as long as the hypothesis in Remark \ref{frac} holds. 
This will be relevant in the discussion below.

We will be working with the quotient: $\oo_K \to \oo_K[y_1, y_2,\dots, y_n]/I^2$. 
Form the corresponding closed subscheme 
$$
{\Aff}^n_{[2]}:= {\Spec}(\oo_K[y_1, y_2,\dots, y_n]/I^2) \subset 
{\Aff}^n={\Spec}(\oo_K[y_1,y_2,\dots, y_n]).
$$

\subsection*{ An embedded group scheme}

\begin{assumption}
\label{Ghyp} 
Let $G$ be a smooth connected quasi-projective group scheme over $S={\Spec}(\oo_K)$. 
Assume further that the conormal bundle module, ${\mathcal N}_G$ is free over $\oo_K$ 
(see Lemma \ref{free1} and Remark \ref{rmk1}).
\end{assumption}

There is a positive integer $n$ and an $\oo_K$-morphism of schemes 
$$G \stackrel{\iota}{\hookrightarrow} \PP^n$$ that is a local immersion, 
identifying $G$ with a locally closed $\oo_K$-subscheme of $\PP^n$.

Denote by ${e}$ the zero-section of $G$ over $\oo_K$, and let ${\mathfrak I}$ 
be the sheaf of ideals in $G$ that cuts out ${e}$. 

Recall the construction $G_{[2]} \subset G$ of \S\ref{ibis}; i.e., the 
subscheme of $G$ cut out by ${\mathfrak I}^2$. 

\begin{proposition}
\label{prop1} 
If ${\bar G} \subset \PP^n$ is the Zariski-closure of the $\oo_K$-subscheme 
$G \subset \PP^n$, then setting 
$$
{\mathcal X}:= {\bar G} \setminus G \subset \PP^n,
$$
the support of ${\mathcal X}$ is disjoint from the zero-section ${e}$.
\end{proposition}
\begin{proof} Since the group scheme $G$ is smooth along the zero-section, the injection $G \hookrightarrow   {\bar G}$ induces an isomorphism on normal bundles along the zero-section, establishing the proposition.\end{proof}
\begin{remark}\label{bar}
 
If, for example, $G=\A$, the N{\'e}ron model of an abelian variety $A$ over $K$, then the 
support of ${\mathcal X}$ in Proposition \ref{prop1} 
is concentrated in fibers of $\A \to {\Spec}(\oo_K)$ 
over the finite punctual subscheme $ \Sigma_{\rm bad}(A) \subset {\Spec}(\oo_K)$ 
where $ \Sigma_{\rm bad}(A)$ is the set of bad primes of $A$, i.e.:
$$ 
\Sigma_{\rm bad}(A):= \sqcup_{\gp | {\rm cond}( A)}\ {\Spec}(k_\gp).
$$
(Here $k_\gp$ is the residue field of the prime $\gp$ of $K$, and 
${\mathrm{cond}}( A)$ is the conductor of $A_{/K}$.)
\end{remark}

\begin{definition}\label{wellarr1} 
An injective (local immersion) $\oo_K$-morphism $\iota: G \hookrightarrow \PP^n$ 
will be called {\df well-arranged} if both of the following properties hold:
\begin{itemize} 
\item 
$\iota$ takes any 
point $P \in G (\oo_K)$ to a point $\iota(P) \in \PP^n$ that can be written 
in homogeneous coordinates $(a_1:a_2:\dots : a_{n+1})$ with $a_i \in \oo_K$, 
and such that the ideal generated by the entries, 
$(a_1,a_2,\dots,a_{n+1}) \subset \oo_K$, is the unit ideal,
\item 
$\iota$ takes the zero-section, ${e} \in G$ to the point $(0:0:\dots:0:1)$ ---this 
being written in homogenous coordinates; i.e., $x_{n+1}=1$:
$$
\xymatrix{{e}\ar[r]^\iota\ar[d] & e\ar[d] & e\ar[d]\ar[l]_= \\
G\ar[r]^\iota & \PP^n & {\rm Aff}^n\ar@{_(->}[l]}
$$ 
\end{itemize}
\end{definition}

\begin{proposition}
Let 
$$
\xymatrix{ S\ar[r]^s\ar[dr]_= & \PP^n\ar[d]\\
& S}
$$
be an $S$-section. Then there is an injective $S$-morphism 
${\mathfrak v}: \PP^n\hookrightarrow \PP^N$ for some $N$ such that 
\begin{enumerate} 
\item 
the image of any $S$-section of $\PP^n$ in $\PP^N$, when written in 
homogenous coordinates $(a_1: a_2:\dots:a_{N+1})$ has the property 
that the ideal generated by the entries, $(a_1,a_2,\dots,a_{N+1}) \subset \oo_K$, 
is a principal ideal in $\oo_K$, and hence after scaling can be taken to be the unit ideal, and 
\item 
${\mathfrak v}(s) =( 0:0:\dots:0:1)$.\end{enumerate}
\end{proposition}

\begin{proof} 
Let $h= h_K$ denote the class number of $K$. Let 
$$
{\mathfrak v}_{n,h}:\PP^n \to \PP^N
$$ 
(with $N:=\binom{n+h}{n}-1$) be the $h$-fold Veronese embedding of 
$\PP^n_{/{\oo_K}}$ in $ \PP^N_{/{\oo_K}}$ (see for example \cite[\S1.4.4.2]{Shaf}), 
i.e., the embedding defined by the rule 
$$
\xymatrix{(x_1,x_2,\dots x_{n+1})~ \ar@{|->}[r]^-{{\mathfrak v}_{n,h}} 
& ~(\mu_1(x_1,\dots,x_{n+1})\dots, \mu_N(x_1,\dots,x_{n+1})),}
$$ 
where the entries of the vector on the right, $\mu_k(x_1,\dots,x_{n+1})$, 
run through the $N+1$ monomials of degree $h$ in the variables $x_1,x_2,\dots, x_{n+1}$.

For any $K$-valued point $\alpha \in \PP^n(K)$ represented by the $n+1$-vector 
$(a_1,a_2,\dots a_{n+1})$ the $\oo_K$-fractional ideal generated by the entries 
$\mu_k(a_1,\dots,a_{n+1})$ of the vector ${\mathfrak v}_{n,h}(\alpha)$ is the 
$h$-th power of the fractional ideal generated by the entries of $\alpha$. 
So this (fractional) ideal is principal.
By scaling our homogenous coordinates by dividing each entry by the inverse 
of a generator of that principal ideal we get assertion (1) of the proposition.

For part (2) of the proposition, let the image of the section 
$s$ (i.e., $ {\mathfrak v}_{n,h}(s) \in \PP^N(\oo_K)$) be represented 
by the point $\tau:=(\tau_1, \tau_2, \ldots,\tau_{N+1}) \in {\Aff}^{N+1}(\oo_K)$ 
where the entries generate the unit ideal. Let $W:= {\Aff}^{N+1}(\oo_K)$ and 
let $T:=\tau\oo_K \subset W$ be the cyclic $\oo_K$-module generated by $\tau$. 
Then $W_0:=W/T$ is a torsion-free $\oo_K$-module, and hence projective, so the 
exact sequence
$$
0 \To T \To W \To W_0 \To 0
$$
splits. Therefore $W \cong W_0\oplus T$, and by the classification theorem for 
projective modules (of finite rank) over Dedekind domains (see for example 
\cite[Theorem II.4.13]{FT}), since $W$ and $T$ are free over $\oo_K$, so is $W_0$.
It follows that we can find an $\oo_K$-basis of $W$ where the first $N$ 
elements of that basis generate $W_0$. That is, there is an $\oo_K$-linear 
change of coordinates of ${\Aff}^{N+1}$ so that after that change ${\Aff}^{N+1}$ 
is given by coordinates $(z_1,z_2,\dots,z_N, z_{N+1})$ where ${\Aff}^N$ is cut 
out by $z_{N+1} = 0$ and the element $\tau$ has coordinates $(0,0,\dots, 0, 1)$.
\end{proof} 

\begin{corollary}
\label{warcor} 
Let $G$ be a group scheme over $S={\Spec}(\oo_K)$ satisfying Assumption \ref{Ghyp}. 
There is a well-arranged injective $S$-morphism $G\stackrel{\iota}{\hookrightarrow} \PP^n$ 
(for some positive number $n$). For any rational point, $P:S \to G$, its image, 
$\lambda(P):=\iota P(S)\in \PP^n$, when written in homogenous coordinates 
$\lambda(P)=(\lambda_1(P),\dots, \lambda_{n+1}(P))$ has the 
property that the entries $\lambda_i(P)$ for $i=1,2,\dots, n$ generate a 
principal ideal in $\oo_K$---equivalently: one can arrange the homogeneous 
coordinates of $\lambda(P)$ by appropriate scalar multiplication so that the 
entries generate the unit ideal. 
\end{corollary} 

\section{The open piece in ${\bar G}$}

From now on, we will fix a quasi-projective group scheme $G$ over $S={\Spec}(\oo_K)$ 
as in Assumption \ref{Ghyp} such that its conformal bundle module ${\mathcal N}_G$ 
is free over $\oo_K$ (using Lemma \ref{free1}) and with a fixed well-arranged 
injective $S$-morphism $G \stackrel{\iota}{\hookrightarrow} \PP^n$ (this being 
guaranteed to exist by Corollary \ref{warcor}).

Recall its Zariski-closure $ G \subset{\bar G} \stackrel{\iota}{\hookrightarrow} \PP^n$ 
as defined in \ref{prop1}. Letting $H \subset \PP^n$ be the hyperplane described 
in \S\ref{aff} above, i.e., cut out by $x_{n+1}=0$. Let $B:={\bar G}\cap H \subset G$ 
denote the divisor in ${\bar G}$ at infinity.

\begin{definition}
Let ${\bar G}_0\subset {\bar G}$ be the Zariski-dense open ($\oo_K$-scheme) 
defined by the cartesian diagram: 
$$
\xymatrix{{e}\ar[r]\ar[d]^= &{\bar G}_0:= \ {\bar G} \cap {\Aff}^n\ar[r]\ar[d]^\iota 
& {\bar G} \setminus B\ar[d]^\iota \ar[r] & G\ar@{^(->}[d]^\iota \\ 
e\ar[r] &{\Aff}^n\ar[r]^= & \PP^n \setminus H\ar[r] & \PP^n}
$$
\end{definition}

The $\oo_K$-scheme ${\bar G}_0$ is an affine scheme, immersed as a closed subscheme 
of ${\Aff}^n$, and contains an open subscheme of the zero-section ${e}$ in $G$. 

The injection $${\bar G}_0\stackrel{\iota}{\hookrightarrow} {\Aff}^n$$ is induced 
by the (surjective) ring homomorphism 
$$
\iota: \oo_K[x_1,x_2,\dots,x_n] \longrightarrow \oo_K[x_1,x_2,\dots,x_n]/(t_1,t_2,\dots,t_m)
$$
where 
$$
t_j(x_1,x_2,\dots,x_n) \in I \subset \oo_K[x_1,x_2,\dots, x_n]; \ {\rm for}\ j=1,2,\dots,m
$$ 
are the polynomials (all with `no constant term') that cut out the affine subscheme 
${\bar G}_0$ in ${\Aff}^n.$ (Recall that $I$ is the ideal generated by the $x_i$.) 

Passing to quotients by $I^2$ we have:
\begin{equation}
\label{gcoord2}
\raisebox{25pt}{\xymatrix{\iota: \oo_K[x_1,x_2,\dots,x_n]\ar[r]\ar[d] 
& \oo_K[x_1,x_2,\dots,x_n]/(t_1,t_2,\dots,t_m)\ar[d] \\
\iota_{[2]}: \oo_K[x_1,x_2,\dots,x_n]/I^2\ar[r] 
& \oo_K[x_1,x_2,\dots,x_n]/(t_1,t_2,\dots,t_m, I^2)} } 
\end{equation}
these being the ring homomorphisms inducing the morphisms of affine schemes:
$$
\xymatrix{{\Aff}^n & {\bar G}_{0}\ar@{_(->}[l] \\ 
{\Aff}^n_{[2]}\ar@{^(->}[u] & {\bar G}_{0,[2]}= G_{[2]}.\ar[l]\ar@{^(->}[u]}
$$

\begin{lemma}
\label{i1} 
After an appropriate $\oo_K$-linear automorphism of the group scheme ${\Aff}^n$ we 
may rewrite the surjective ring homomorphism of \eqref{gcoord2} 
\begin{equation}
\label{i2}
\iota_{[2]}: \oo_K[x_1,x_2,\dots,x_n]/I^2\to \oo_K[x_1,x_2,\dots,x_n]/(t_1,t_2,\dots,t_m, I^2)
\end{equation}
that induces the group scheme morphism 
$G_{[2]}\stackrel{\iota}{\hookrightarrow} {\Aff}^n_{[2]}$
as the projection
$$
\iota_{[2]}: \oo_K[x_1,x_2,\dots,x_n]/(x_1,x_2,\dots,x_n)^2\to 
\oo_K[x_1,x_2,\dots,x_d]/(x_1,x_2,\dots,x_d)^2
$$
where $d$ is the dimension of the group scheme $G_{/K}$, and the mapping 
$\iota_{[2]}$ is given by:
\begin{itemize} 
\item 
$x_i \mapsto x_i$ if $i \le d$, and 
\item 
$x_i \mapsto 0$ if $d < i \le n$.
\end{itemize} 
\end{lemma} 

\begin{proof} 
This uses the fact that the group scheme $G$ is smooth, its conormal bundle 
module ${\mathcal N}_G$ is free over $\oo_K$, and the injection 
$G\hookrightarrow \PP^n$ is well-arranged.  
To be explicit, consider the ideal $J:=(t_1,t_2,\dots,t_m, I^2)$, 
so we may rewrite \eqref{i2} above as 
$$
\iota_{[2]}: \oo_K[x_1,x_2,\dots,x_n]/I^2\to \oo_K[x_1,x_2,\dots,x_n]/J,
$$ 
and we can find $n-d$ generators for the free $\oo_K$-module $J/I^2$. 
Letting $\{t_j; \ j=1,2,\dots, n-d\} \subset J$ be lifts of those generators, the ideal 
$J:=(t_1,t_2,\dots,t_{n-d}, I^2)$ as positioned in the sequence of ideals
$$
I^2 = (x_1,x_2,\dots,x_n)^2 \subset J=(t_1,t_2,\dots,t_{n-d}, I^2) 
\subset I=(x_1,x_2,\dots,x_n) \subset \oo_K
$$
has the property that $J/I^2$ is a free $\oo_K$-module with $n-d$ generators\newline  $\{t_j; \ j=1,2,
\dots, n-d\}$ and  is---as (free) $\oo_K$-submodule of  the free $\oo_K$-module 
$I/I^2$---a direct summand.  That is, there is a free $\oo_K$-submodule 
${\mathcal U} \subset I/I^2$ such that:
$$
I'^2 = J/J^2\oplus  {\mathcal U}.
$$
It follows that after a linear change of variables (over $\oo_K$) 
we can arrange it so that $t_i \equiv x_{d+i }\ \mod \ I^2$ for $i=1, 2, \dots, n-d$. 
\end{proof} 

\begin{corollary}
\label{arrange} 
Keeping to the above notation, and the terminology of \S\ref{ibis} we have:
\begin{enumerate} 
\item 
$
{\mathcal R}_G= \oo_K[x_1,x_2,\dots,x_d]/(x_1,x_2,\dots,x_d)^2,
$
\item 
$
\displaystyle
{\mathcal N}_G= \bigoplus_{i=1}^d {\bar x}_i\oo_K
$
\\
where ${\bar x}_i$ is the image of $x_i$ in $\oo_K[x_1,x_2,\dots,x_d]/(x_1,x_2,\dots,x_d)^2$,
\item 
$
\displaystyle
{\mathcal N}_G^*= \bigoplus_{i=1}^d {\bar x}_i^*\oo_K
$
\\
where ${\bar x}_i^*: {\mathcal N}_G\to \oo_K$ is the ring homomorphism sending 
${\bar x}_i$ to $1\in \oo_K$ and ${\bar x}_j$ to $0$ if $j \ne i$.
\end{enumerate}
\end{corollary}

\section{Vanishing and congruence ideals}
\label{VC}

Recall that we have fixed a group scheme $G$ satisfying Assumption \ref{Ghyp} 
above. Denote by $e:S \to G$ its identity-section. To say that $G$ is quasi-projective 
means the structure morphism $G \to S$ is a quasi-projective morphism (see 
\cite[Definition 5.3.1]{EGA}) hence is of finite type, and since $S$ is an 
affine noetherian scheme, $G \to S$ is a morphism of finite presentation.  
Let $P$ be  an $\oo_K$-point  of $G$.

\begin{definition}
By the {\df vanishing ideal of $P$} we mean the ideal $z_P\subset \oo_K$ 
defining the intersection of the $S$-section $P$ with the identity section. 

By the {\df congruence ideal of $P$} we mean the ideal $c_P\subset \oo_K$ 
defining the intersection of the $S$-section $P$ with the subscheme 
$G_{[2]} \subset G$. 

That is, the ideals $z_P$ and $c_P$ are the ideals that fit into diagram 
\eqref{cart1} below where the rectangles are cartesian. 
\begin{equation}
\label{cart1}
\raisebox{48pt}{\xymatrix{{\Spec}(\oo_K/z_P)\ar@{^(->}[r]\ar@{^(->}[d] 
&{\Spec}(\oo_K/c_P)\ar@{^(->}[r]\ar[d]^{P_{[2]}} \ \ & S\ar[d]^P\\
S\ar[r]^{e}\ar[d]^\simeq & G_{[2]}\ar@{^(->}[r]\ar[d]^\simeq &G\\
{\Spec}(\oo_K)\ar@{^(->}[r] & {\Spec}({\mathcal R}_G) }}\end{equation}
\end{definition}

\begin{lemma}
\label{cz} 
\begin{enumerate}
\item 
The ideal $c_P$ is the square of the ideal $z_P$. 
\item 
For $P={ e}$, the identity section, we have that $z_{ e}$ and $c_{e}$ 
are equal to $(0)$. If $P \ne { e}$ then $z_P$ and $c_P$ are nonzero ideals.
\item 
Let $Q,P$ be $\oo_K$-valued points of $G$. Then (writing the group law of $G$ 
multiplicatively) $z_{Q\cdot P}$ is contained in the ideal $(z_Q , z_P)$ 
generated by $z_Q$ and $z_P$.
\end{enumerate}
\end{lemma}

\begin{proof} 
\begin{enumerate}
\item 
Recalling that the identity section $S \stackrel{{e}}{\hookrightarrow} G$ 
is the subscheme cut out by the sheaf of ideals ${\mathcal I} \subset \oo_G$, 
and $ G_{[2]} \stackrel{e}{\hookrightarrow} G$ is the subscheme cut out by 
the square of that sheaf of ideals ${\mathcal I}^2 \subset \oo_G$, 
cartesian-ness of \eqref{cart1} then implies that ideal $c_P$ is the 
square of the ideal $z_P$. 
\item 
The cartesian square 
$$
\xymatrix{{\Spec}(\oo_K/z_{e})\ar@{^(->}[r]\ar@{^(->}[d] & S\ar[d]^{P={ e}}\\
S\ar[r]^{e} & G}
$$
gives us that ${\Spec}(\oo_K/z_{e})=S$; so $z_{ e}=(0)$, and hence 
$c_{ e} = z_{ e}^2=(0)$ as well.
\item 
Briefly: note that $P$ and $Q$ restricted to the base 
${\Spec}(\oo_K/(z_P, z_Q))$ are both equal to the identity section 
(over that base), so their product is as well. 
\end{enumerate} 
\end{proof}

\begin{lemma}
\label{cz1} 
For $\nu \in \N$, let $Q = P^\nu$. Then $z_Q \subset z_P$.
\end{lemma}

\begin{proof} 
Induction on $\nu$, using Lemma \ref{cz}(3).
\end{proof}

\section{Coordinates of rational points}

Assume now that we have arranged coordinates as in Lemma \ref{i1} and Corollaries \ref{arrange}  
and \ref{warcor}. The $\oo_K$-rational point $P$ 
of $G$ and its image $\iota(P) \in \PP^n$ fit into a diagram
$$
\xymatrix{ S= {\Spec}(\oo_K)\ar[r]^-{P}& G\ar[r]^\iota &\PP^n\\
{\Spec}(\oo_K/c_P)\ar@{^(->}[u]\ar[r] & G_{[2]}\ar@{^(->}[u]\ar[r]^{\iota_{[2]}} 
& {\Aff}^n_{[2]}\ar[u]\\
{\Spec}(\oo_K/z_P)\ar@{^(->}[u]\ar[r] & {e}\ar@{^(->}[u]\ar[r]^\simeq & e\ar@{^(->}[u].}
$$
and the image $\iota(P):=\iota P(S) \in \PP^n$ has homogeneous coordinates 
$$
\lambda(P) = \big( \lambda_1(P): \lambda_2(P): \dots :\lambda_n(P):\lambda_{n+1}(P) \big) 
\in {\Aff}^{n+1}(\oo_K)
$$
where the elements $$
\{\lambda_1(P), \lambda_2(P),\dots, \lambda_n(P), \lambda_{n+1}(P)\}\subset \oo_K
$$
generate the unit ideal. In particular, the ideal generated by 
$\{\lambda_1(P), \dots, \lambda_n(P)\}$ and the principal 
ideal generated by $\lambda_{n+1}(P)$ are relatively prime ideals in 
$\oo_K$. Also, using Lemma \ref{i1} and Corollary \ref{arrange} we may assume that 
$\lambda_j(P) = 0$ for $d<j \le n$.

\begin{proposition}
The vanishing ideal $z_P$ is the ideal generated by 
$$
\{\lambda_1(P),\dots, \lambda_n(P)\}.
$$
\end{proposition} 

\begin{proof} 
This follows directly from the definitions.
\end{proof}

If 
$$
{\mathbf a}:=(\lambda_1:\lambda_2:\dots:\lambda_n:\lambda_{n+1}) \ \in \ \oo_K^{n+1}
$$ 
are homogenous coordinates for the point $P$ with properties described  
above, then letting $\delta(P):= \lambda_{n+1}$ (i.e., the denominator), 
with notation defined in \eqref{pa} above we may write:
$$
{\mathbf a}^! = (\frac{\lambda_1(P)}{\delta(P)}, 
\frac{\lambda_2(P)}{\delta(P)},\dots,\frac{\lambda_n(P)}{\delta(P)}) 
\in \big(\oo_K[{\frac{1}{\delta(P)}}]\big)^n 
$$

Write the morphism $P_{[2]}: {\Spec}(\oo_K/c_P) \rightarrow G_{[2]}$ 
as a homomorphism of the underlying affine rings, 
\begin{equation}
\label{p2}
P_{[2]}: \oo_K\oplus {\mathcal N}_G \;= \; {\mathcal R}_G \to \oo_K/c_P.
\end{equation}
Consider the $\oo_K$-dual of the locally free $\oo_K$-module ${\mathcal N}_G$
\begin{equation}
\label{dual}
{\mathcal N}^* = {\mathcal N}_G^* = \Hom_{\oo_K}({\mathcal N}_G, \oo_K)
\end{equation}

\begin{definition}
If $P$ is an $S$-section of $G_{/S}$ let 
$\partial P \in {\mathcal N}^*\otimes_{\oo_K}(\oo_K/c_P)$ denote the element 
(in ${\mathcal N}^*\otimes_{\oo_K}(\oo_K/c_P)\ = \ {\mathcal N}^*\otimes_{\oo_K}(\oo_K/z_P^2)$) 
determined by the $\oo_K$-homomorphism $P_{[2]}$ restricted to $ {\mathcal N}_G$ in \eqref{p2}.
\end{definition}

Using the quotation marks below to indicate passing to the quotient 
$$
\oo_K[{\frac{1}{\delta(P)}}] \to \oo_K/z_P^2
$$ 
we get
$$
{\mathbf a}_{[2]}^! =``\biggl(\frac{\lambda_1(P)}{\delta(P)}, 
\frac{\lambda_2(P)}{\delta(P)},\dots,\frac{\lambda_d(P)}{\delta(P)}, 0,0,\dots 0\biggr)" 
\in (\oo_K/z_P^2)^n,
$$
and using Lemma \ref{i1} and \eqref{i2}, we get that coordinates for 
$\partial P$ are given by:
\begin{equation}
\label{eq3}
\partial P =``\biggl(\frac{\lambda_1(P)}{\delta(P)}, 
\frac{\lambda_2(P)}{\delta(P)},\dots,\frac{\lambda_d(P)}{\delta(P)}\biggr)" 
\in z_P ({\mathcal N}_G^*\otimes_{\oo_K} \oo_K/z_P^2) 
\subset ({\mathcal N}_G^*\otimes_{\oo_K} \oo_K/z_P^2).
\end{equation}

\begin{definition} If $R$ is a commutative ring, and $W$ is a free $R$-module 
of finite rank, an element $w \in W$ is called a {\df basis element} if any 
of the following equivalent properties hold:
\begin{itemize} 
\item 
$W/(Rw)$ is a free $R$-module;
\item 
$w$ is a member of a basis of the free $R$-module $W$;
\item 
if $I$ is any nonunit ideal of $R$ then $w \notin I W$. 
\end{itemize} 
\end{definition}

\begin{lemma}
\label{app} 
Let $P$ be a section of $G$ that is {\em not} the identity section, so 
$z_P \ne 0$. Suppose, as well, that $z_P$ is not the unit ideal. Let
$$
{\mathcal W}_P:={\mathcal N}^*\otimes_{\oo_K}(z_P/z_P^2) 
\subset {\mathcal N}^*\otimes_{\oo_K}(\oo_K/z_P^2).
$$

Then
\begin{enumerate}
\item 
$\partial P \equiv 0 \pmod{{\mathcal N}^*\otimes_{\oo_K}{z_P}}$, so 
$\partial P \in {\mathcal W}_P$,
\item 
the $\oo_K/z_P$ module ${\mathcal W}_P$ is free over $\oo_K/z_P$, 
\item 
$\partial P$ is a basis element of ${\mathcal W}_P$. 
\end{enumerate} 
\end{lemma}

\begin{proof}
Assertions (1) and (3) follow directly from the definition of vanishing ideal, 
while assertion (2) follows from the fact that ${\mathcal N}^*$ is locally free 
over $\oo_K$. 
\end{proof} 

\begin{remark} 
With notation as in Lemma \ref{app}, if $P,Q$ are $S$-sections with property 
that $z_Q \subset z_P$ we have the natural $\oo_K$-module homomorphism
$$
{\mathcal W}_Q \xrightarrow{~\iota_{(P,Q)}~} {\mathcal W}_P.
$$
\end{remark}

\begin{proposition}
\label{sum} 
$\partial(P\cdot Q) \equiv \partial(P)+ \partial( Q) \pmod{(c_Q , c_P)}$. 
\end{proposition} 

\begin{proof} 
The mapping $(P\cdot Q)_{[2]}: {\mathcal R}_{G} \to \oo_K$ is given by the 
composition of the maps 
$$
{\mathcal R}_G \stackrel{\gamma_{[2]}}{\longrightarrow} 
{\mathcal R}_{G\times_{\oo_K} G} \stackrel{\phi}{\longrightarrow} \oo_K
$$
where $\phi$ restricted to $ {\mathcal N}_G\otimes_{\oo_K} 1$ is $P_{[2]} \otimes_{\oo_K}1$,
and $\phi$ restricted to $1\otimes_{\oo_K} {\mathcal N}_G$ is $1\otimes_{\oo_K} Q_{[2]}$. 
The result follows from Proposition \ref{propcomm}.
\end{proof} 

\begin{corollary}
\label{qp} 
Let $Q= P^\nu$. Then 
$$
\partial Q = \nu\cdot \partial P \in {\mathcal W}_P.
$$
\end{corollary} 

\begin{proof} This follows from Lemmas \ref{cz}(1) and \ref{cz1} and Proposition \ref{sum}. 
\end{proof}

\part{Proof of the main theorems}

We will keep to the convention that a field of algebraic 
numbers that is allowed to have infinite degree over ${\Q}$ will be put in 
boldface, e.g., ${\bL, \bK}$, but if it is assumed to be a number field, 
i.e., a field of finite degree over $\Q$, it will be in normal type, e.g., $L, K$.

\section{Capturing subrings by congruences}
\label{sec: cong}

\begin{lemma}
\label{le:congruence}
Let $L/K$ be an extension of number fields, $\oo_L/\oo_K$ their corresponding 
rings of integers, and let $M/ K$ be the Galois closure of $L/K$. Let 
$\alpha \in \oo_L$ and $b \in \oo_K$. Suppose there exists an ideal 
$I \subset \oo_K$ with the following properties: 
\begin{gather}
\label{norm:ineq}
\text{$|{\mathbf N}_{M/\Q}(I\oo_M)| > |{\mathbf N}_{M/\Q}(\alpha-\beta)|$ 
for every conjugate $\beta$ of $\alpha$ over $K$,}\\
\label{eq:equiv}
\alpha \equiv b \pmod {I\oo_L}.
\end{gather}
Then $\alpha \in \oo_K$.
\end{lemma}

\begin{proof}
Since $I \subset \oo_K$, it follows from \eqref{eq:equiv} that every conjugate $\beta$ of $\alpha$ 
satisfies $\beta \equiv b \pmod{I\oo_M}$. Therefore $\alpha -\beta \equiv 0 \pmod{I\oo_M}$. Consequently,
\[
{\mathbf N}_{M/\Q}(\alpha -\beta) \equiv 0 \pmod{{\mathbf N}_{M/\Q}(I\oo_M)}
\]
in $\Z$. Thus, either $\alpha=\beta$ or 
\[
|{\mathbf N}_{M/\Q}(\alpha -\beta)| \geq |{\mathbf N}_{M/\Q}(I\oo_M)|
\]
which contradicts \eqref{norm:ineq}.
\end{proof}

\subsection*{Norm inequalities }

While it is clear that a congruence like \eqref{eq:equiv} can be rewritten as 
a divisibility condition in the ring of integers assuming we are given generators 
of the ideal $I$, it is not a priori clear how to convert \eqref{norm:ineq} into 
a polynomial equation with variables ranging over that ring. The propositions 
below explain how it can be done.

\begin{definition}\label{C} For $m$ a positive integer  let $C(m)$ denote $(m+1)^2$ times  the smallest positive integer greater than 
the maximum absolute value of any minor  of the matrix 
\begin{equation}\label{matrix} \left(\begin{array}{cccc} 
0&0&\ldots & 1 \\
1&1& \ldots & 1\\
\ldots\\
m^m&m^{m-1}&\ldots&1
\end{array}\right).\end{equation}
\end{definition}

\begin{proposition}
\label{prop:bounds}
Let $M/{\Q}$ be a  Galois number field extension of  degree $m$.
Suppose $\alpha \in \oo_M$ and define
\begin{equation}\label{u}
u= u(m, \alpha):=C(m)\cdot \alpha(1-\alpha)\cdots (m-\alpha). 
\end{equation}
Then for every conjugate $\beta$ of $\alpha$ over $\Q$ and every ideal $I$ of 
$\oo_M$ contained in $u^{m^2}\oo_L$, we have
$$
| {\mathbf N}_{M/\Q}(\alpha-\beta)| < |{\mathbf N}_{M/\Q}(I\oo_M)|.
$$
\end{proposition}

\begin{proof}

Let $g(T)$ be the characteristic polynomial of $\alpha$ over $\Q$ as an element of $M$. 
Let $g(T)=T^m +a_{m-1}T^{m-1}+ \ldots + a_0$ with $a_i \in \Z$.
For every $r \in \Z$ we have
\begin{equation}
\label{eq:Ng}
{\mathbf N}_{M/\Q}(r-\alpha)=g(r).
\end{equation}
Put $C:=C(m)$.
By definition of $u$, if $1 \le r \le m$, then
\begin{equation}
\label{eq:w_r}
r-\alpha=\frac{u}{C\prod_{i=1,i\ne r}^m(i-\alpha)}.
\end{equation}
Put
\[
w_r:=\prod_{i=1,i\ne r}^m(i-\alpha).
\]
Then from \eqref{eq:Ng} and \eqref{eq:w_r} it follows that 
\[
r^m+a_{m-1}r^{m-1} +\ldots +a_0 = 
\frac{1}{C^m{\mathbf N}_{M/\Q}(w_r)}{\mathbf N}_{M/\Q}(u)=c_r{\mathbf N}_{M/\Q}(u), 
\]
where $|c_r| \leq \frac{1}{C^m} $.
Now consider the following linear system
$$
\left (\begin{array}{cccc} 
0&0&\ldots & 1 \\
1&1& \ldots & 1\\
\ldots\\
m^m&m^{m-1}&\ldots&1
\end{array}\right )
\left (\begin{array}{c}
1\\
a_{m-1}\\
\ldots\\
a_0
\end{array}
\right )= \left (\begin{array}{c}
c_0{\mathbf N}_{M/\Q}(u)\\
c_1{\mathbf N}_{M/\Q}(u)\\
\ldots\\
c_m{\mathbf N}_{M/\Q}(u)
\end{array}
\right ).
$$
Using Cramer's rule we obtain that $a_i=\frac{D_i}{D}$, where
\[
D_i=\left |\begin{array}{cccccc} 
0&0&\ldots& c_0{\mathbf N}_{M/\Q}(u)&\ldots & 1 \\
1&1& \ldots&c_1{\mathbf N}_{M/\Q}(u)  & \ldots & 1\\
\ldots\\
(m)^m&(m)^{m-1}&\ldots&c_m{\mathbf N}_{M/\Q}(u) &\ldots&1
\end{array}\right |,
\]
with the column $(c_r{\mathbf N}_{M/\Q}(u)), r=0,\ldots, m$ replacing the $i$-th column of the  matrix  \ref{matrix}  and $D$ is the discriminant of \ref{matrix}.  Factoring out ${\mathbf N}_{M/\Q}(u)$ and expanding along the $i$-th column we obtain the following:
\[
\left |\frac{D_i}{{\mathbf N}_{M/\Q}(u)}\right |=\left |\begin{array}{cccccc} 
0&0&\ldots& c_0&\ldots & 1 \\
1&1& \ldots&c_1& \ldots & 1\\
\ldots\\
m^m&m^{m-1}&\ldots&c_m &\ldots&1
\end{array}
\right |=\left |\sum_{r=0}^m \pm c_rM_{r,i}\right |\leq \sum_{r=0}^m |c_r||M_{r,i}|,
\]
where $M_{r,i}$ is the minor of \ref{matrix} corresponding to the elimination of the $r$-th row and the $i$-th column. By assumption 
\[
|M_{r,i}| < \frac{C}{(m+1)^2},
\]
and by construction $|c_r| \leq \frac{1}{C^m}$.  Thus,
for all $i=0,\ldots,m$, we have that 
\[
|D_i| <\frac{1}{(m+1)C^{m-1}}|{\mathbf N}_{M/\Q}(u)|. 
\]
 Finally, taking into account that $|D| >1$, we conclude that
\[
|a_i| < |D_i|<\frac{1}{(m+1)C^{m-1}}|{\mathbf N}_{M/\Q}(u)|.
\]
Suppose now that some root $\gamma$ of $g$ is greater in absolute value than 
$|ma_r|$ for all $r=0,\ldots, m-1$. Let $|a_{\rm max}|=\max_{0\leq i \leq m-1}{|a_i|}$.  In this case we have that
\[
|\gamma^m |=|-a_{m-1}\gamma^{m-1} - \ldots -a_0|\leq \sum_{r=0}^{m-1}|a_r\gamma^r|\leq |ma_{\rm max}\gamma^{m-1}| < |\gamma^m |,
\]
and we have a contradiction. Thus, every root $\gamma$ of $g$ satisfies
\[
|\gamma| < \frac{1}{C^{m-1}}|{\mathbf N}_{M/\Q}(u) |\le \frac{1}{2}|{\mathbf N}_{M/\Q}(u)|,
\]
and for any two roots $\gamma$ and $\delta$ of $g$ we have that
\[
|\gamma-\delta| < |{\mathbf N}_{M/\Q}(u)|
\]
(we may assume $m > 1$ or there is nothing to prove).
Now
\[
|{\mathbf N}_{M/\Q}(\alpha-\beta)|< |{\mathbf N}_{M/\Q}(u)^{m(m-1)}|<|{\mathbf N}_{M/\Q}(I\oo_M)|.
\]
\end{proof}

\begin{lemma}
\label{le:Dbound} 
Let $M$ be a finite Galois  totally real extension of $\Q$,  and put $m := [M:\Q]$. 
Further, let $x\in \oo_M$ be such that $|\sigma(x)| > 1$ for every embedding $\sigma:M \hookrightarrow \R$.
Then for every $\gamma \in \Gal{M}{\Q}$ we have
$$
|{\mathbf N}_{M/\Q}(x-\gamma(x))| \leq 2^m{\mathbf N}_{M/\Q}(x^2).
$$
\end{lemma}

\begin{proof}  Let $\sigma, \tau: M \hookrightarrow \R$ be embeddings. 
Since $|\sigma(x)|, |\tau(x)|>1$ we have that 
$$
|\sigma(x)-\tau(x)| < 2|\sigma(x)\tau(x)|, 
$$
so
\begin{multline*}
|{\mathbf N}_{M/\Q}(x-\gamma(x))| =\prod_{\sigma}|\sigma(x)-(\sigma\circ\gamma)(x)| \\
<\prod_{\sigma}2|\sigma(x)(\sigma\circ\gamma)(x)| 
=2^m\prod_{\sigma }|\sigma(x)|^2
=2^m{\mathbf N}_{M/\Q}(x^2).
\end{multline*}
\end{proof}


\begin{corollary}
\label{cor:totreal}
Let $\alpha$ be an algebraic integer contained in some totally real Galois extension $M/\Q$ and such 
that $\sigma(\alpha)>1$ for every embedding $\sigma : M \hookrightarrow \R$. 
Let $I$ be an ideal of $M$ such that $I\subset (2\alpha+1)^2\oo_{M}$.

Then for any conjugate $\beta$ of $\alpha$ over $\Q$ we have that 
$|{\mathbf N}_{M/\Q}(\alpha-\beta)| \leq |{\mathbf N}_{M/\Q}(I\oo_M)|$.
\end{corollary}

\begin{proof}
Let $v := 2\alpha+1$ and $m:=[M:\Q]$.
For every embedding $\sigma: M \hookrightarrow \R$ we have
$$
2 \leq 2\sigma(\alpha) \leq \sigma(v),
$$
and therefore 
$$
2^m \leq 2^m{\mathbf N}_{M/\Q}(\alpha) \leq {\mathbf N}_{M/\Q}(v).
$$
By Lemma \ref{le:Dbound}, we now have 
$$
|{\mathbf N}_{M/\Q}(\alpha-\beta)| \leq 2^m{\mathbf N}_{M/\Q}(\alpha^2)
< {\mathbf N}_{M/\Q}(v^2) \leq {\mathbf N}_{M/\Q}(I).
$$
\end{proof}

\subsection*{The basic congruence relation for extensions of finite degree and totally real fields} 

Fix $L/K$  
an extension of number fields and let $m$ denote the degree of  $M$, the Galois closure of $L/{\Q}$.

 For a positive integer $m$ and an algebraic integer $\alpha \in \oo_L$ define 
$$
\label{eq:Dalpha}
D(m, \alpha):= u(m, \alpha)^{m^2}=(C\alpha(1-\alpha)\cdots (m-\alpha))^{m^2} \in \oo_L, 
$$ 
where $C=C(m)$ is as in Definition \ref{C} and   $ u(m, \alpha)$ is as in Definition \ref{u} .

Putting together Lemma \ref{le:congruence}, Proposition \ref{prop:bounds} and 
Corollary \ref{cor:totreal} we immediately obtain these corollaries.

\begin{corollary}
\label{CONG}

 Let $E\subset \oo_L$ be the set of all 
elements $\alpha \in \oo_L$ such that there exists $b \in \oo_K$ and an ideal 
$I \subset \oo_K$ satisfying 
$$
I\oo_L \subset D(m, \alpha)\oo_L, \quad \alpha \equiv b \pmod{I\oo_L}.
$$
Then $ E \subset \oo_K$.
\end{corollary} 
\begin{corollary}
\label{CONG2}
Let ${\bL}$ be a totally real extension of \   $\Q$, possibly infinite. 
Let $K \subset {\bL}$ be a number field. 
Let $E\subset \oo_{\bL}$ be the set of all elements $\alpha \in \oo_L$ such that
there exists $b \in \oo_K$, an ideal $I \subset \oo_K$, and 
$u_1, u_2, u_3, u_4 \in \oo_{\bL}$ satisfying 
$$
I\oo_{\bL} \subset (2\alpha+1)^2\oo_{\bL}, \quad \alpha=1 + u_1^2 + \ldots +u_4^2,
\quad \alpha \equiv b \pmod{I\oo_{\bL}}.
$$
Then $ E \subset \oo_K$.
\end{corollary}
\begin{remark}
The equations for the totally real case are the same across all pairs of 
totally real fields $K$ and ${\bL}$, including the case where one 
or both extensions are infinite. We will show that the same is true in 
the case of a quadratic extension of a totally real field.
\end{remark}

\section{Quadratic extensions of totally real fields}
We treat separately quadratic extensions of totally real fields because if the totally real field is of infinite degree over $\Q$, this case is technically much more complicated than the case of finite extensions or the case of totally real fields.  The main reason for the complications is the difficulty with bounds on norms.
\subsection*{Norm inequalities for quadratic extensions of totally real fields}
The construction of diophantine bounds on norms of elements of a non-totally 
real quadratic extension of a totally real field of infinite degree over $\Q$ 
is not as simple as the analogous constructions in the case of extensions of 
finite degree over $\Q$ or totally real fields of arbitrary degree.  
We construct a diophantine definition of these bounds in several steps starting 
with Lemma \ref{le:newbounds} below and continuing with Lemma \ref{le:always big}, 
Corollaries \ref{cor:existsW},  \ref{cor:algg}, \ref{cor:Diophstable} and \ref{cor:ext2}.  
Unlike diophantine definitions of bounds in the other two cases, in the case of 
non-totally real quadratic extensions of totally real fields of infinite degree over $\Q$, 
we will need to use diophantine stability of a commutative group scheme in the 
extension ${\mathbf F}/{\bL}$, where ${\bL}$ is a totally real field possibly of 
infinite degree over $\Q$ and ${\bF}$ is a non-totally real quadratic extension 
of ${\bL}$.  The group scheme, a twist of $\gm$, is constructed in \S\ref{sec:quad}.

\begin{proposition}
\label{le:newbounds}
Let $\bf{L}$ be a possibly infinite totally real extension of $\Q$ and let ${\bF}$ 
be a quadratic extension of ${\bL}$. Let $\delta \in \oo_{\bF}$, 
$\delta^2=d \in \oo_{\bL}$, ${\bF}={\bL}(\delta)$. Let 
$x=y_0+y_1\delta \in \oo_{\bF}$  with $y_0, y_1 \in {\bL}$, and let $w  \in {\bL}$ be such that 
\begin{align}
\label{eq:conds}
&\text{for all embeddings $\sigma: {\bf F} \hookrightarrow \R$ 
we have that $1 < \sigma(x) < \sigma(w)$}, \\
\label{eq:condt}
&\text{for all embeddings $\tau: {\bf F} \hookrightarrow \C$ 
with $\tau({\bf F})  \not\subset \R$ we have that $|\tau(w) |\geq 1$.} 
\end{align}

Let $L \subset \bf{L}$ be a number field containing $d, y_0,y_1$ and $w$. 
Let $F=L(\delta)$. Then $|{\mathbf N}_{F/\Q}(\delta y_1)| \leq |{\mathbf N}_{F/\Q}(xw)|$.
\end{proposition}
\begin{remark}
Recall that the inequality involving $\sigma$ can be converted to an equation by Proposition \ref{tpprop}.
\end{remark}
\begin{proof}
If $\tau$ is a non-real embedding of $F$, then $|\tau(y_1\delta)| \leq |\tau(x)|$. 
If $\sigma$ is a real embedding of $F$, then let $\hat \sigma$ be an embedding  of 
$F$ such that $\sigma_{|L}=\hat \sigma_{|L}$ but $\hat\sigma \ne \sigma$. In other words, if 
$\sigma(x)=\sigma(y_0)+\sigma(\delta)\sigma(y_1)$, then $\hat \sigma(x)=\sigma(y_0)-\sigma(\delta)\sigma(y_1)$. 
Then either 
\[
|\sigma(x)|=|\sigma(y_0)|+|\sigma(\delta y_1)|,
\]
 or 
\[
|\hat \sigma(x)|=|\sigma(y_0)|+|\sigma(\delta y_1)|.
\] 
So, either 
$|\hat \sigma(x) | \geq |\hat \sigma (\delta y_1)|$ or $|\sigma(x)|\geq |\sigma(\delta y_1)|$. 
Since $\sigma(w)=\hat \sigma(w)$, for all real embeddings $\sigma$ of $F$ we have that 
$|\sigma(w)| > |\sigma(\delta y_1)|$.
Let $\Sigma$ be the collection of all real embeddings of $F$, and let $T$ be the 
collection of all embeddings of $F$ that are not real. Now we have the following inequalities:
\begin{align*}
|{\mathbf N}_{F/\Q}(y_1\delta )| 
&= \prod_{\tau \in T}|\tau(\delta y_1)|\prod_{\sigma \in \Sigma}|\sigma(\delta y_1)|\\
&\leq \prod_{\tau \in T}|\tau(x)|\prod_{\sigma \in \Sigma}|\sigma(wx)|\\
&\leq \prod_{\tau \in T}|\tau(x)|\prod_{\sigma \in \Sigma}|\sigma(x)|\prod_{\tau \in T}|
\tau(w)|\prod_{\sigma \in \Sigma}|\sigma(W)|= |{\mathbf N}_{F/\Q}(xw)|.
\end{align*}
\end{proof}

For Proposition \ref{le:newbounds} to be useful we need to be provided with an existentially defined  bound ``$w$" that satisfies \eqref{eq:conds} and \eqref{eq:condt}.  This is Corollary \ref{cor:existsW} below.  In  preparation, we have:

\begin{lemma}
\label{le:always big}
Let $K$ be a totally real number field. Let 
$\Sigma=\{\sigma_1, \ldots, \sigma_n\}$ be the collection of all 
embeddings of $ K$ into $\R$.
 Let $\Omega \subset K$ be an infinite set. 
Then for any integer $N>0$ there exists $u \ne v \in \Omega$ such that for all 
$\sigma \in \Sigma$ we have that
\[
\left|\sigma \left(\left [\left(\frac{1}{u-v}\right )^2+1\right](u^2+1) \right )\right | > N.
\]
\end{lemma}

\begin{proof}
As an infinite subset of $\R$, for each embedding $\sigma$ of $K$ into $\R$, the 
set $\sigma(\Omega)$ is either unbounded or it has a limit point 
and therefore contains a non-constant Cauchy sequence. Assume without loss of 
generality that $\sigma_1(\Omega)$ contains a Cauchy sequence $\{\sigma_1(u_i)\}$ such that all elements 
of the sequence are distinct. If $\{|\sigma_2(u_i)|\}$ is unbounded, then select 
a subsequence $\{u_{i,2}\}$ such that $\{|\sigma_2(u_{i,2})|\} \rightarrow \infty$. 
If $\{\sigma_3(u_{i,2})\}$ is bounded, then let $\{u_{i,3}\}$ be a subsequence of 
$\{u_{i,2}\}$ such that $\{\sigma_3(u_{i,2})\}$ is a Cauchy sequence. Continuing 
by induction we construct a sequence $\{u_{i,n}\}$ such that for all 
$\sigma \in \Sigma$ we have that $\{\sigma(u_{i,n})\}$ is either a Cauchy 
sequence or is going to infinity. Now let $N>0$ be given. Choose $j \in \Z_{>0}$ 
such that for all $\sigma \in \Sigma$ and all $\ell \geq j$ either 
$|\sigma(u_{\ell,n})|>N$ or $|\sigma(u_{j,n}-u_{\ell,n})| < \frac{1}{N}$. 
We claim that for all $\sigma \in \Sigma, \ell > j$ it is the case that
\[
\left|\sigma \left(\left [\left(\frac{1}{u_{j,n}-u_{\ell,n}}\right )^2+1\right]
(u_{j,n}^2+1) \right )\right | > N.
\]
Indeed, for each $\sigma \in \Sigma$ we have that either $|\sigma(u_{j,n})| > N$ or 
\[
\left | \sigma\left (\frac{1}{u_{j,n}-u_{\ell,n}}\right ) \right |> N
\]
while the $\sigma$-images of both factors are always bigger than 1.
\end{proof}

\begin{corollary}
\label{cor:existsW}
Let ${\bL\subset \bF}$ be as in Proposition \ref{le:newbounds}. Let $\Omega \subset {\bL}$ 
contain infinitely many elements of some number field $K \subset {\bL}$. 
    Then for 
any $x \in \oo_{\bF}$ there exist $u, v \in \Omega$ such that 
\[
w:=\left [\left(\frac{1}{u-v}\right )^2+1\right](u^2+1)
\]
satisfies \eqref{eq:conds} and \eqref{eq:condt}.
\end{corollary}

\subsection*{A special case of stability for quadratic extensions of totally real fields}
\label{sec:quad}
As we have mentioned above, to construct diophantine bounds on norms of elements of non-totally real extensions of degree 2 of totally real fields, we will need to use a particular case of diophantine stability.  We discuss this case in this section.  More specifically, we consider diophantine stability of multiplicative groups
over rings of integers in  quadratic extensions ${\bF}/{\bL}$, where ${\bL}$ 
is a totally real possibly infinite algebraic extension of $\Q$. 

J. Denef and L. Lipshitz were the first to use this phenomenon for the purposes of 
existential definability over finite extensions (see \cite{Den2}).  The 
third author used it over infinite extensions in conjunction with 
diophantine stability of elliptic curves (see \cite{Sh37}). 

Here we 
present a different proof of these results using the vocabulary of 
diophantine stability. Our goal is to show that this case is of the 
same nature as other instances of diophantine stability already 
discussed in this paper. We begin by considering number fields and 
then move to infinite extensions.

\subsection*{The case of finite extensions}
Suppose $M/L$ is a quadratic extension of fields. We denote by $\mathbf{G}_m^{M/L}$ 
the twist of the multiplicative group $\gm$ over $L$ by the quadratic character 
corresponding to $M/L$, as defined for example in \cite{MRS}. 
If $F/L$ is a field extension and $F \cap M = L$, then 
\begin{equation}
\label{dutw}
\mathbf{G}_m^{M/L}(\oo_F) = \{x \in \oo_{MF}^\times : \mathbf{N}_{MF/F}x = 1\},
\end{equation}
\begin{equation}
\label{dutw2}
\mathbf{G}_m^{M/L}(\oo_L) = \{x \in \oo_{M}^\times : \mathbf{N}_{M/L}x = 1\}.
\end{equation}

\begin{lemma}
\label{min2}
Suppose $L$ is a totally real number field, and $F$ is 
a quadratic extension of $L$.
Suppose $M/L$ is a quadratic extension such that for every infinite 
place $v$ of $L$, $v$ ramifies in $M/L$ if and only if $v$ does {\em not} ramify 
in $F/L$. Then 
\begin{enumerate}
\item
$[\mathbf{G}_m^{M/L}(\oo_F):\mathbf{G}_m^{M/L}(\oo_L)]$ is finite,
\item
if $F$ is not totally real, then $\mathbf{G}_m^{M/L}(\oo_L)$ has elements 
of infinite order,
\item
if $n$ is an integer, $n \ge 3$, then 
$\ker\{\mathbf{G}_m^{M/L}(\oo_F) \to \mathbf{G}_m^{M/L}(\oo_F/n\oo_F)\} \subset \mathbf{G}_m^{M/L}(\oo_L)$.
\end{enumerate}
\end{lemma}

\begin{proof}  Consider the diagram of fields in the hypothesis of the lemma.
$$
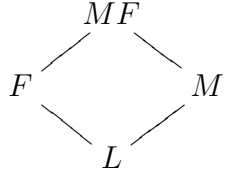
\xymatrix@C12pt@R12pt{\ & MF & \ \\ F\ar@{-}[ur] & \ & M\ar@{-}[ul] \\
& L\ar@{-}[ul]\ar@{-}[ur]}
$$
Let  $V$ be the set of archimedean places of $L$  (all real, by hypothesis). 
Write $V = V_F\sqcup V_M$ where $V_F\subset V$  is the subset consisting of the 
places that  do not ramify in $F/L$ and $V_M$  is the subset consisting of the 
places that  do not ramify in $M/K$. So 
$$
[L:{\Q}]  =|V| = |V_F| + |V_M|.
$$
We have that  $r_F :=2\cdot  |V_F|$ is the number of real places of $F$  
and  $r_M:=2\cdot  |V_M|$ is the number of real places of $M$. Letting  
$s_F, s_M$ denote the number of complex places of $F$ and $M$ respectively, 
we have: $s_F= |V_M|$ and $s_M= |V_F|$.  As for $MF/L$ we have that $MF$ is 
totally complex and every archimedean place of $L$ lifts to two complex 
places of $MF$.   Letting $u_K$ denote the rank of the group of units of 
a field $K$, we have, by Dirichlet's Unit Theorem:  
$$
\begin{array}{cc}
u_L= |V|-1, & u_M =2|V_M|+ |V_F|-1, \\[7pt]
u_F = 2|V_F|+|V_M|-1,& u_{MF}= 2|V|-1
\end{array}
$$
so that
$$
u_{MF}-u_F= |V_M| = u_M-u_L.
$$

Let $G$ denote the group $\mathbf{G}_m^{M/L}$.  
The above combined with \eqref{dutw}, shows that
\begin{gather}
\notag \rank_\Z G(\oo_F) = u_{MF}-u_F  = |V_M|, \\
\label{rk} \rank_\Z G(\oo_L) = u_M-u_L = |V_M|.
\end{gather}
so that
$\rank_\Z G(\oo_F) = \rank_\Z G(\oo_L)$. This proves (1).

Equation \eqref{rk} shows that $\rank_\Z G(\oo_L) > 0$ 
unless $|V_M|=0$, i..e., unless  $F$ is totally real. This proves (2).

Suppose now that $x \in \ker\{G(\oo_F) \to G(\oo_F/n\oo_F)\}$. 
Using \eqref{dutw} we can view $x \in \oo_{FM}^\times$ such that 
$x \equiv 1 \pmod{n}$.
By (1), there is a positive integer $k$ such that $x^k \in \oo_M^\times$. 
If $\sigma$ is the nontrivial automorphism of $MF/M$, then $(x/x^\sigma)^k = 1$, 
so $x/x^\sigma$ is a root of unity. But $x/x^\sigma \equiv 1 \pmod{n}$, 
so we have $x/x^\sigma = 1$, i.e., $x \in M$. Since 
$\mathbf{N}_{M/L}x = \mathbf{N}_{MF/F}x = 1$, we have $x \in G(\oo_L)$ by \eqref{dutw}.
This proves (3).
\end{proof}

\begin{lemma}
Suppose $L$ is a totally real number field, and $F$ is 
a quadratic extension of $L$.
Then there is a quadratic extension $M/L$ such that for every infinite 
place $v$ of $L$, $v$ ramifies in $M/L$ if and only if $v$ does {\em not} ramify 
in $F/L$.
\end{lemma}

\begin{proof}
Choose $d\in L$ such that $F = L(\sqrt{d})$ and 
let $M = L(\sqrt{-d})$.
Then for every infinite place $v$ of $L$, $-\alpha$ is negative 
at $v$ if and only if $\alpha$ is positive at $v$, so 
$v$ ramifies in $M/L$ if and only if $v$ doesn't ramify in $F/L$.
\end{proof}

\subsection*{The case of quadratic extensions totally real fields ${\bL}$ of infinite degree.}
Let ${\bL}$ be a totally real algebraic extension of $\Q$. Let ${\bF}$ and 
${\mathbf M}$ be quadratic extensions of ${\bL}$ such that for every embedding 
$\sigma$ of $\mathbf{MF}$ into $\bar \Q$ we have that $\sigma(F) \subset \R$ 
if and only if $\sigma(M) \not \subset \R$. Let ${\bF}={\bL}(\delta)$ where 
$\delta^2=d \in \oo_{\bL}$ and let ${\mathbf M}={\bL}(\beta)$ where 
$\beta^2=-d \in \oo_{\bL}$. We let $G := \mathbf{G}_m^{{\mathbf M}/{\bL}}$ as above.

\begin{lemma}
\label{le:Diophstability}
if $n$ is an integer, $n \ge 3$, then 
$\ker\{G(\oo_{\bF}) \to G(\oo_{\bF}/n\oo_{\bF})\} \subset G(\oo_{\bL})$.
\end{lemma}

\begin{proof}
Suppose 
$x \in \ker\{G(\oo_{\bF}) \to G(\oo_{\bF}/n\oo_{\bF})\}$. 
Let $L$ be a number field contained in ${\bL}$ such that $d \in L$ and 
$x \in G(\oo_{L(\delta)})$. Let $F=L(\delta)$ and $M=L(\beta)$. Then 
$x \in \ker\{G(\oo_{F}) \to G(\oo_F/n\oo_{ F})\} \subset \mathbf{G}_m^{M/L}(\oo_{ L})$
by Lemma \ref{min2}(3).
By \eqref{dutw2} we have 
$\mathbf{G}_m^{M/L}(\oo_{L}) \subset \mathbf{G}_m^{\mathbf{M/L}}(\oo_{\bL})$, 
so $x \in G(\oo_{\bL})$.
\end{proof}

\begin{corollary}
\label{cor:algg}
Let ${\bL}$ be a totally real algebraic extension of $\Q$. Let ${\bF}$ 
be a quadratic extension of ${\bL}$ and assume that ${\bF}$ is not 
totally real. Then there exists a commutative group scheme $G$ defined over ${\oo_\bL}$ 
such that
\begin{enumerate}
\item
$G(\oo_\bL)$ contains an element of infinite order,
\item
for every integer $n \ge 3$, 
$\ker\{G(\oo_{\bF}) \to G(\oo_{\bF}/n\oo_{\bF})\} \subset G(\oo_{\bL})$.
\end{enumerate}
\end{corollary}

\begin{proof}
By Lemmas \ref{min2}(2) and \ref{le:Diophstability}, the commutative group 
scheme $G := \mathbf{G}_m^{\mathbf{M/L}}$ has these properties.
\end{proof}

\begin{remark}
In the language of Definition \ref{ids} below, Corollary \ref{cor:algg}(2) says that 
if $n \ge 3$ then 
$(1,n\oo_\bL)$ is an exponent of diophantine stability for $G$ relative to $\bF/\bL$.
\end{remark}

There is another consequence of diophantine stability we will use later to 
produce bounds for elements of ${\bF}$. 

\begin{corollary}
\label{cor:Diophstable}
Let ${\bF}/{\bL}$ be as in Corollary \ref{cor:algg}. Then there exists a set $B \subset \oo_{\bF} \times (\oo_{\bF}\setminus \{0\})$ satisfying the following conditions.
\begin{enumerate}
\item $B$ is diophantine over $\oo_{\bF}$.
\item If $(a,b) \in B$, then $a/b \in {\bL}$.
\item For some number field $L \subset {\bL}$, the set
$
\{a/b : \text{$(a,b) \in B$ and $a/b \in L$}\}
$
is infinite.
\end{enumerate}
\end{corollary}

\begin{proof}
Let  $G$  be a  group scheme satisfying the conclusions of Corollary \ref{cor:algg}. 
Let $G \hookrightarrow \mathbb{P}^n$ be a well-arranged embedding  (see Definition \ref{wellarr1})
$$
P \mapsto \big(x_1(P): x_2(P):\dots : x_{n+1}(P)\big).
$$ 
Let $A$ be the set of all $(n+1)$-tuples in $\oo_{\bF}\times \cdots \times \oo_{\bF}$ 
that are homogeneous coordinates of some point 
$P \in \ker\{G(\oo_{\bF}) \to G(\oo_{\bF}/n\oo_{\bF})\}$.
Then $A$ is diophantine over $\oo_\bF$.  It follows that the set 
$$
B := \{(y_i,y_j) : 0 \le i, j \le n+1, y_j \ne 0, (y_1,\ldots y_{n+1}) \in A\}
$$
is diophantine over $\oo_\bF$ as well.  By Corollary \ref{cor:algg}(2), if $(a,b) \in B$ 
then $a/b \in \bL$.

By Corollary \ref{cor:algg} we can fix a point $P$ of infinite order in $G(\oo_\bL)$.  
For some positive $k$ we have $P^k \in \ker\{G(\oo_{\bF}) \to G(\oo_{\bF}/n\oo_{\bF})\}$.  
Then $P^k \in G(\oo_L)$ for some number field $L \subset \bL$, and for every integer $m$ 
we have 
$$
P^{km} \in \ker\{G(\oo_{\bF}) \to G(\oo_{\bF}/n\oo_{\bF})\} \cap G(\oo_L),
$$
and it follows that the set in assertion (3) of the corollary is infinite.
\end{proof}

\subsection*{Combining bound equations and basic congruence equations in the case of 
non-totally real quadratic extensions of totally real fields}
In the following corollary we combine Proposition \ref{le:newbounds}, Lemma \ref{le:always big},  Corollaries \ref{cor:existsW} and \ref{cor:Diophstable} to finish our construction of norm bounds.  We remind the reader that inequalities for real embeddings are implemented using sums of squares, pairs in $B$ come from a totally real field, and the role of $u$ and $v$ is explained in Corollary \ref{cor:existsW}.
\begin{corollary}
\label{cor:ext2}
Let ${\bF}$ be a quadratic extension of a totally real field ${\bL}$. 
Let $B$ be as in Corollary \ref{cor:Diophstable}. Let $\alpha \in \oo_{\bF}$. 
Let $\delta\in \oo_{\bF}$ be such that ${\bF}={\bL}(\delta)$ 
and $\delta^2:=d \in \oo_{\bL}$. Let $H \subset {\bL}$ be a number 
field such that $d \in H$, $\alpha \in H(\delta)$ and let $\hat \alpha$ be the 
conjugate of $\alpha$ over $H$. Let $\Sigma$ be the set of all real embeddings 
of the field $\bF$. Consider now the following equations and conditions: 
\begin{gather*}
(a,b), (c,d) \in B,\\
u=a/b, v=c/d, bd \ne 0\\
 X_2 \ne 0 \\
X_1= X_2\left [\left(\frac{1}{u-v}\right )^2+1\right](u^2+1) \\
\forall \sigma \in \Sigma: |\sigma(X_2)| <|\sigma(X_2\alpha)| < |\sigma(X_1)|.
\end{gather*}
If this system of equations and conditions is satisfied over $\oo_{\bF}$, then 
\[
|{\mathbf N}_{H(\delta, u,v)/\Q}(\alpha-\hat \alpha)| \leq |{\mathbf N}_{H(\delta, u,v)/\Q}(2X_1 \alpha)|. 
\]
Conversely, for any $\alpha \in \oo_{\bF}$, this system can be satisfied.
\end{corollary}
\begin{proof}
First, assume the equations and the conditions in the statement of the corollary are satisfied.  Then by construction of $B$ we have that $u, v \in {\bL}$.  Let $W=\frac{X_1}{X_2}\in L$ and observe that for any embedding $\mu: \bF \hookrightarrow \C$ we have that $\mu(W)>1$.  Further, for every $\sigma \in \Sigma$ we have that $1<|\sigma(\alpha)| <\sigma(W)$.  If we let $\alpha=y_0+\delta y_1$ and  $\hat \alpha=y_0-y_1\delta$, then 
\[
|{\mathbf N}_{H(\delta,u,v)/\Q}(\hat \alpha -\alpha)|=|{\mathbf N}_{H(\delta,u,v)/\Q}(2\delta y_1)| \leq |{\mathbf N}_{H(\delta,u,v)/\Q}(2\alpha W)|\leq |{\mathbf N}_{H(\delta,u,v)/\Q}(2\alpha X_1)|, 
\]
where the penultimate inequality is true by Proposition \ref{le:newbounds}.

Suppose now that $\alpha \in \oo_F$, then by construction of $B$, Corollary \ref{cor:Diophstable} and Corollary \ref{cor:existsW} we can find $(a,b)$ and $(c,d)$ in $B$ to satisfy the equations and conditions of the corollary.
\end{proof}
Finally, we combine the bounds in Corollary \ref{cor:ext2} with the basic 
congruence condition (Lemma \ref{le:congruence}) for the case of a non-totally real quadratic extension.

\begin{corollary}
\label{cor:extension2}
Let ${\bL}$ be a totally real field. Let ${\bF}$ be a quadratic extension 
of ${\bL}$. Let $K$ be a number field contained in ${\bF}$. Let $\Sigma$ be the set of all 
embeddings of ${\bF}$ into $\R$. Let $T$ be the set of all non-real embeddings of ${\bf F}$ into $\C$. 

Let $E \subset \oo_{\bf F}$ be the set of elements 
$\alpha \in \oo_{\bF}$ such that there exist elements 
$ X_1 \in \oo_{\bF}, 0\ne  X_2  \in \oo_{\bF}$ with $\frac{X_1}{X_2} \in \bL$,  $b \in \oo_K,$ and $I$ an ideal of $\oo_{K}$ satisfying the following conditions:
\begin{align*}
\forall \sigma \in \Sigma: \, \, 1 <\sigma(\alpha) < \sigma \left (\frac{X_1}{X_2} \right ),\\
\forall \tau  \in T: \, \tau\left (\frac{X_1}{X_2}\right ) \geq 1,\\
I\oo_{\bF} \subset 2X_1\alpha\oo_{\bF}\\ 
\alpha \equiv b \pmod{I\oo_{\bF}}.
\end{align*}
Then $ E \subset \oo_K$.  (Here we again remind the reader that for real embeddings we can convert inequalities to equations via Proposition \ref{tpprop}.)
\end{corollary}

\section{Rational points} 
Let $G$ be a group scheme satisfying Assumption \ref{Ghyp} above. 

\begin{definition}
Let 
$$
\xymatrix{ S\ar[r]^P\ar[dr] & G\ar[d]\\
& S}
$$
be an $S$-section of $G$. We allow ourselves a number of synonyms for this 
notion. If the ring $\oo_K$ rather than the corresponding scheme $S$ is 
more prominent in the surrounding discussion, we may also call an $S$-section 
an $``\oo_K$-section," or an $``\oo_K$-rational point." In the case, for 
example, when the group scheme $G:=\A$ is the N{\'e}ron model over 
$S={\Spec}(\oo_K)$ of an abelian variety $A_{/K}$, the $S$-sections of 
$G$---alias $\oo_K$-sections of $\A$---are in one-one correspondence with the 
$K$-rational points of the abelian variety $A$ over $K$ to which these sections restrict.
\end{definition}

\begin{definition} Set $M:= G(S)$. The letter ``$M$" is for {\em Mordell-Weil group}, 
a label we use even if $G$ is any of the group schemes that we work with---i.e., as in 
Assumption \ref{Ghyp}; we may also denote $M$ as $G(\oo_K)$. 
\end{definition} 

\begin{definition}
\label{mi} 
For any algebraic extension $L$ of $\Q$ and any nonempty subset $I\subset \oo_L$ 
different from $\{0\}$ let 
$$ 
M_{I,L}(G):= \ker\{G(\oo_L) \longrightarrow G(\oo_L/I\oo_L)\}
$$
where $I\oo_L$ denotes the (nonzero) ideal of $\oo_L$ generated by $I$.
If $r \in {\Z}_{>0}$ define
$$
M^r_{I,L}\;:= \;\{x^r : x \in M_{I,L}\}.
$$
That is, we have a surjection
$$
M_{I,L} \xrightarrow{\text{$r$-th power}} M^r_{I,L}\subset M_{I,L}.
$$ 
\end{definition}

\begin{lemma}
The group $M_{I,L} = M_{I,L}(G)$ is a subgroup of $M = G(\oo_L)$ of finite index.
\end{lemma}

\begin{proof} 
Since $I \oo_L$ is a nonzero ideal, $\oo_L/I \oo_L$ is a finite ring, so 
$G({\Spec}(\oo_L/I\oo_L))$ is a finite group. 
\end{proof}

\begin{definition}
\label{ids} 
Say that a pair ($r,I)$ consisting of a positive integer $r$ together with a nonzero ideal 
$I\subset \oo_K$ is an {\df exponent of diophantine stability for 
$G$ relative to a field extension $L/K$} if the subgroup 
$M_{I,L}^r(G)\subset G(\oo_L)$ is contained in $G(\oo_K)$:
$$
\{x^r : x \in M_{I,L}\} \;\subset\; G(\oo_K).
$$
\end{definition}

\begin{example}
Property (3) of Lemma \ref{min2} says that, in the notation of that lemma, 
if $n \ge 3$ then the pair $(1,n\oo_K)$ 
is an exponent of diophantine stability for $\mathbf{G}_m^{M/L}$ relative to $F/L$.
\end{example}

\begin{remark} 
If the subgroup $G(\oo_K)$ is of finite index $m$ in $G(\oo_L)$ then $(m, \oo_L)$ 
is an exponent of diophantine stability for $G$ relative to $L/K$.
If $G_{/K}=A$ is an abelian variety, then there exists an exponent of 
diophantine stability for $G$ relative to $L/K$ if and only if $A$ is 
{\em rank stable} for the field extension $L/K$.
\end{remark}

\begin{lemma}
\label{serretate}
Suppose $A$ is an abelian variety over $K$, rank stable for the extension $\bL/K$.  
Then there is a positive integer $n$ such that $(1,n\oo_K)$ is an exponent of 
diophantine stability for $A$ relative to $\bL/K$.
\end{lemma}

\begin{proof}
Let $\bL'$ denote the Galois closure of $\bL$ over $K$.
Let $\gp, \qq$ be primes of $\bL'$ where $A$ has good reduction, and with 
distinct residue characteristics $p, q$, respectively.
By \cite[Lemma 2]{Se-Ta}, reduction modulo $\gp$ is injective on prime-to-$p$ torsion,
and reduction modulo $\qq$ is injective on prime-to-$q$ torsion.  
Hence the only torsion in $M_{\gp,\bL'}$ is $p$-power torsion, and 
the only torsion in $M_{\qq,\bL'}$ is $q$-power torsion. 
Thus, setting $n = pq$, the torsion subgroup of $M_{n,\bL'}$ is zero.

Now suppose $P \in M_{n,\bL}$.  Since $A$ is rank stable for $\bL/K$, there 
is a positive integer $t$ such that $t P \in A(K)$.  If $\sigma \in \Gal{\bar{\Q}}{K}$, 
then $t(\sigma P - P) = \sigma (tP) - tP = 0$.  Therefore $\sigma P - P$ is a torsion 
point in $M_{n,\bL'}$, so $\sigma P - P = 0$ and we conclude that 
$\sigma P = P$.  Since this holds for every 
$\sigma \in \Gal{\bar{\Q}}{K}$, we have $P \in A(K)$.  This proves the lemma.
\end{proof}

\begin{remark}
If $(r,I)$ is an exponent of diophantine stability for $G$ relative to $L/K$, 
we have the diagram 
$$
\xymatrix{ M^r_{I,L}\ar@{^(->}[r]\ar@{^(->}[rd] & M_{I,L}\ar@{^(->}[r] & G(\oo_L)\\
& G(\oo_K)\ar@{^(->}[ru] & \ }.
$$
\end{remark}

\begin{lemma}
If $G_1$ and $G_2$ are two smooth group schemes over $\oo_K$ and $(r,I)$ is 
an exponent of diophantine stability for both $G_1$ and $G_2$ relative to a 
field extension $L/K$, then $(r,I)$ is an exponent of diophantine stability 
for $G_1\times_{\oo_K}G_2$ relative to $L/K$.
\end{lemma}

\begin{lemma}
\label{dsinf} 
Suppose that $G(\oo_{\bK})$ contains a point of infinite order. Then for every 
nonzero ideal $I\subset \oo_{\bK}$, every $r \in \Z_{>0}$, and every ${\bL}/{\bK}$, 
the group $M^r_{I,{\bL}}(G)$ coontains a point of infinite order.
\end{lemma}

\begin{proof} 
Let $L$ be a number field contained in $\bL$, let $P \in G(\oo_L)\subset G(\oo_{\bL})$ be a point of infinite order. Let $N:= |G(\oo_L/(I\cap \oo_L)\oo_L)|.$ 
Then $N < \infty$, $P^N$ is contained in $ M_{I \cap \oo_L ,L} \subset M_{I ,{\bL}}$ and $P^{r N}$ is 
contained in $M^r_{I,\bL}$. Since $P^{r N}$ has infinite order, this proves the lemma.
\end{proof}

\begin{lemma}
\label{zD2} 
Let $I\subset \oo_K$ be any nonzero ideal, and let $P$ be an $S$-section of $G$ 
in $M_{I,K}$. Then $z_P \subset I.$
\end{lemma}

\begin{proof} 
This follows directly from Definition \ref{mi}: if $P \in M_{I,K}(G)$ then the 
image of $P$ in $G(\oo_K/I)$ is trivial.
\end{proof}

\begin{proposition}
\label{maincong2} 
Let $L/K$ be a finite extension, with $\oo_L/\oo_K$ their corresponding rings 
of integers. Let $m:=$ the degree of the Galois closure $M/{\Q}$  of $L$ over ${\Q}$. Let 
$C:=C(m)$ be  as in Proposition \ref{prop:bounds}.

Fix $\alpha \in \oo_L$, and let 
$ D(m,\alpha):=(C\alpha(1-\alpha)\ldots (m-\alpha))^{m^2} \in \oo_L$.
Suppose
\begin {itemize} 
\item 
$z \subset \oo_K$ is an ideal such that $z \oo_L\subset D(\alpha) \oo_L$, 
\item 
$W$ is a free $\oo_K/z$-module of finite rank.
\end{itemize}
If either $D(m,\alpha)=0$ or if there are elements $v, w \in W$ such that 
$w \in W$ is a basis element and 
\begin{equation}
\label{basic} 
v \otimes 1 = w \otimes \alpha = \alpha(w \otimes 1) \in W \otimes_{\oo_K}\oo_L
\end{equation} 
then $\alpha \in \oo_K.$ 
\end{proposition}

\begin{proof} 
Suppose that $\alpha, v,w$ satisfy \eqref{basic}. Let $z_L := z\oo_L$. 
Since $w$ is a basis element of the free $\oo_K/z$-module $W$, we see that 
$w \otimes 1$ is a basis element of the free $\oo_L/z_L$-module 
$W \otimes_{\oo_K} \oo_L$.

All of \eqref{basic} `takes place' in the free $\oo_L/z_L$-submodule 
of $W \otimes_{\oo_K} \oo_L$ (of rank one) generated by $w \otimes 1$. In 
particular $v \in (\oo_K/z) w$ so we can choose $b \in \oo_K$ such that $v = b w$. 
Then, using \eqref{basic},
$$
w \otimes \alpha = v \otimes 1 = bw \otimes 1 = w \otimes b.
$$
Since $w \otimes 1$ is a basis vector, it follows that $b \equiv \alpha \pmod{z_L}$. 
The proposition then follows from Corollary \ref{CONG}.
\end{proof}               

In the same manner, using Corollaries \ref{CONG2} and \ref{cor:extension2}, 
one can prove a totally real version and a quadratic extension 
of a totally real field version of Proposition \ref{maincong2}. 

\begin{proposition}
Let ${\bL}/\bK$ be an extension of totally real fields, with $\oo_{\bL}/\oo_{\bK}$ 
their corresponding rings of integers. Fix $\alpha \in \oo_{\bL}$ such that 
$\alpha=1+v_1^2 +\ldots +v_4^2$ with $v_i \in \oo_\bL$.   Suppose
\begin {itemize} 
\item 
$z \subset \oo_{\bK}$ is an ideal such that $z \oo_{\bL}\subset  (2\alpha+1)^2\oo_{\bL}$, 
\item 
$W$ is a free $\oo_{\bK}/z$-module of finite rank.
\end{itemize}
If there are elements $v, w \in W$ such that $w \in W$ is a basis element and 
$$
v \otimes 1 = w \otimes \alpha = \alpha(w \otimes 1) \in W \otimes_{\oo_{\bK}}\oo_{\bL}
$$ 
then $\alpha \in \oo_{\bK}.$ 
\end{proposition}

\begin{proposition}
\label{prop:Talpha}
Let
$$
\xymatrix{ {\bL} \ar@{^(->}[r] & {\bF}\\
& {\bK}\ar@{^(->}[u]}
$$
be a diagram of fields of algebraic numbers where $\bK$ is an algebraic possibly infinite extension of $\Q$, 
${\bL}$ is totally real, and $[{\bF}:{\bL}] =2$. For elements 
$\alpha \in \oo_{\bF}$ and $X_1, X_2 \in \oo_{\bF}$ with $X_2$ nonzero, 
putting $X := \frac{X_1}{X_2}$ suppose that: 
\begin{enumerate}
\item $X \in {\bL}$,
\item 
for every embedding $\tau: {\bF} \hookrightarrow \C$, we have 
$\tau(X) \in \R$ and $\tau(X) >1$,
\item 
for every real embedding $\tau: {\bF} \hookrightarrow {\R}$ , we have $1 <\tau(\alpha) < \tau(X)$ 
(this inequality can be rewritten as an equation by Proposition \ref{tpprop}), 
\end{enumerate} 
Suppose further that $z \subset \oo_{\bK}$ is an ideal such that 
$z \oo_{\bF}\subset X_1\alpha \oo_{\bF}$ and there is a free $\oo_K/z$-module 
of finite rank $W$ for which there are elements $v, w \in W$ such that $w \in W$ 
is a basis element and $$
v \otimes 1 = w \otimes \alpha = \alpha(w \otimes 1) \in W \otimes_{\oo_{\bK}}\oo_{\bF}.
$$
Then $\alpha \in \oo_{\bK}.$ 
\end{proposition}
Before stating the corollary below we recall the definition of ${\mathcal W}_P$ from Lemma \ref{app}:
\[
{\mathcal W}_P:={\mathcal N}^*\otimes_{\oo_K}(z_P/z_P^2) 
\subset {\mathcal N}^*\otimes_{\oo_K}(\oo_K/z_P^2).
\]
\begin{corollary}
\label{cor01} 
Let $(r,I)$ be an exponent of diophantine stability for $G$ relative to 
$L/K$ (resp. ${\bL}/{\bK}$, ${\bF}/{\bK}$). Let $Y$ denote the set of $\alpha \in \oo_L$ 
for which there are points $P, Q \in M^r_{ID(\alpha),L}$ 
(resp. $M^r_{(2\alpha+1)^2I,{\bL}}$, $M^r_{IX_1\alpha,{\bF}}$) with $P \ne {e}$ such that 
\begin{equation}
\label{eqcor1} 
\partial Q = \alpha\cdot \partial P\;\in\; {\mathcal W}_P\otimes_{\oo_K}\oo_L 
(\text{ resp. } {\mathcal W}_P\otimes_{\oo_{\bK}}\oo_{\bL}, {\mathcal W}_P\otimes_{\oo_{\bK}}\oo_{\bF}). 
\end{equation} 
In the totally real case assume $\alpha=1+u_1^2 +\cdots + u_4^2$. 
In the case of a quadratic extension let $X_1$ be defined as in 
Proposition \ref{prop:Talpha}. Then: 
\begin{enumerate} 
\item 
$ Y\subset\oo_K$ (resp. $Y \subset \oo_{\bK}$).
\item 
If $G(\oo_K)$ (resp. $G(\oo_{\bK})$) contains a point of infinite order, 
then $\N \subset Y \subset \oo_K$  (resp. $\N \subset Y \subset \oo_{\bK}$).
\end{enumerate} 
\end{corollary} 

\begin{proof} 
We consider the case of a number field extension $L/K$ first.  If $D(m,\alpha) = 0$ 
then $\alpha \in \{0,1,2,\dots,m\}$, so we may suppose that $D(m,\alpha) \ne 0$.
By the diophantine stability hypothesis of this corollary we have 
$P,Q \in G(\oo_K)$---in particular these are $K$-rational points. 

To connect with the notation of Proposition \ref{maincong2} above, 
let $W := {\mathcal W}_P$ and $z := z_P$, noting that ${\mathcal W}_P$ 
is a free $\oo_K/z_P$-module by Lemma \ref{app}(2). 
Take the $v$ and $w$ of Proposition \ref{maincong2} to be, respectively, the 
images of $\partial P $ and $ \partial Q$ in ${\mathcal W}_P$. Since 
$P \in M^r_{D(m,\alpha)I,L} \subset M_{D(\alpha)I,L}$ we have that 
$z_P \subset D(m,\alpha)I\oo_L$. Proposition \ref{maincong2} then gives us that 
$Y$ is contained in $\oo_K$, which is (1). 

Now suppose $\alpha \in \N$. Find a nontrivial point $P \in M^r_{D(m,\alpha)I,L}$. 
Note that such a $P$ exists by Lemma \ref{dsinf} since $G(\oo_K)$ contains a point 
of infinite order. Let $Q:= P^\alpha$. By Corollary \ref{qp} we have
$$\partial Q = \alpha\cdot \partial P \in {\mathcal W}_P$$
so $\alpha \in Y$. This proves (2).

In the case of a totally real extension or quadratic extension of a totally 
real field, we note that $(2\alpha+1)^2$ and $X_1$ cannot be 0. Next we 
proceed exactly as above replacing $D(m,\alpha)$ by $(2\alpha+1)^2$ in the totally real case and replacing $D(m,\alpha)$ by $X_1\alpha I$, with $n$ being an integer greater or equal to 3, in the case of a quadratic extension of a totally real field. 
\end{proof}
 
The case ${\bK=\bL}$ is of particular interest to us, so we add the following corollary.

\begin{corollary}
\label{cor: FtoK}
Let $\bF, \bL$ be as above with ${\bF}$ not totally real and let $G$ be the twist of $G_m$ defined at the beginning of Section \ref{sec:quad}.  Let $(r,I)=(1,3\oo_F)$. Let $Y$ denote the set of $\alpha \in \oo_{\bF}$ 
for which there are points $P, Q \in M^1_{3X_1\alpha\oo_{\bF},{\bF}}$ with $P \ne {e}$ such that 
$$
\partial Q = \alpha\cdot \partial P\;\in\; {\mathcal W}_P\otimes_{\oo_{\bL}}\oo_{\bF}. 
$$
Here $X_1$ is again defined as in Proposition \ref{prop:Talpha}. Then $\N \subset Y \subset \oo_{\bL}$.
\end{corollary}

\begin{proof}
The corollary follows from Proposition \ref{prop:Talpha}, Lemmas \ref{min2} and \ref{le:Diophstability}.
\end{proof}
\begin{remark}
While $D(\alpha)$ and $(2\alpha +1)^2$ are obviously polynomial in nature, 
we remind the reader that one can see  that $X_1, X_2$ can be described in 
a diophantine fashion from Corollaries \ref{cor:Diophstable} and \ref{cor:ext2}. 
\end{remark}

\begin{remark}
Whether or not there are any points of infinite order in $G(\oo_K)$, our proof 
will show that a given natural number $\nu$ is in $E$ as long as $M^r_{D(\nu),L}$ 
(resp. $M^r_{D(\nu),{\bL}}$) is not trivial.
\end{remark}

\section{An existential formulation of Corollaries \ref{cor01} and \ref{cor: FtoK} }
\label{sec:exist}

Our aim is to give a formulation of Corollaries \ref{cor01} and \ref{cor: FtoK} 
entirely in the language of $\oo_{\bL}$. Assume that $G$ is a group scheme over $\oo_K$ satisfying 
Assumption \ref{Ghyp}. Then:

\begin{enumerate}
\item 
There exists a system of homogenous equations over $\oo_K$ that defines 
${\bar G}$ the Zariski closure of $G$ in $\PP^n$ as described in 
Proposition \ref{prop1} above, so using Lemma \ref{3.3} above, there is 
an existential definition of the set of homogenous coordinates 
$(\lambda_1:\lambda_2:\dots:\lambda_{n+1}) \in \Aff^{n+1}(\oo_L) $ 
that are representatives of points $P \in {\bar G}(\oo_L )$. 
\item 
Hence if, for example, $G={\mathcal A}$ is an abelian scheme, $G$ is an 
open $\oo_L$-subscheme in ${\bar G}$ defined by a finite set of local congruences 
(see Remark \ref{bar}) so there is an existential definition of the set of homogenous 
coordinates $$(\lambda_1:\lambda_2:\dots:\lambda_{n+1}) \in \Aff^{n+1}(\oo_L) $$ 
that are representatives of points $P \in { G}(\oo_L )$. 
\item 
Given two points $P,Q\in G(\oo_L)$ add a further set of variables 
$$
\{r_1,r_2,\dots,r_{n+1}; s_1,s_2,\dots, s_{n+1}\}
$$ 
and the equations (in the rings $\oo_L[r_1,r_2,\dots, r_{n+1}]$ and $\oo_L[s_1,s_2,\dots, s_{n+1}]$ respectively):
\begin{equation}
\label{hPQ} 
{\mathcal E}_P: \quad \sum_{i=1}^{n+1} \ r_i \lambda_i(P) = 1, \qquad
{\mathcal E}_Q: \quad \sum_{i=1}^{n+1} \ s_i \lambda_i(Q) = 1.
\end{equation} 
This augmented system of equations gives us an existential definition of the 
sets of homogenous coordinates that generate the unit ideal for $P$ and for 
$Q$, a pair of points in $G(\oo_L)$.
\item 
Given an ideal $J \subset \oo_L$ defined by an explicit finite set of 
generators $J:=(j_1,j_2,\dots,j_t) \subset \oo_L$, we have an existential 
definition of the sets of homogenous coordinates that generate the unit 
ideal for $P$ and for $Q$, a pair of points in $M_{J,L}$ since the subgroup 
$M_{J,L}\subset G(\oo_L)$ is defined by explicit congruence conditions.
\item 
Since the $r$-th power mapping 
$$
G \xrightarrow{\text{$r$-th power}} G^r \subset G
$$ 
is defined by a system of equations over $\oo_L$ it then follows that 
$M^r_{J,L}$ has an existential definition in terms of $M_{J,L}$.
\item 
Note that if $(r,I)$ is an exponent of diophantine stability for $L/K$ then if 
$P\in M^r_{J,L} \subset G(\oo_K)$ is such that set of homogenous coordinates 
$$
\{\lambda_i(P) : 1 \le i \le n+1\} \subset \oo_L
$$ 
generates the unit ideal in $\oo_L$ then there exists a unit $u\in \oo_L^\times$ such that
$$
\{u\lambda_i(P) : 1 \le i \le n+1\} \subset \oo_K \subset \oo_L.
$$ 
\end{enumerate} 

A consequence of this discussion is: 

\begin{corollary} 
\label{exist}
Assume given: 
\begin{itemize} 
\item 
a group scheme $G$ satisfying Assumption \ref{Ghyp},
\item 
a well-arranged embedding $\iota:G \hookrightarrow \PP^n$,
\item 
a finite extension $L/K$
\item 
an exponent $(r,I)$ of diophantine stability for $L/K$, and
\item 
a set of generators $j_1, j_2,\dots, j_t \in \oo_K$ of $I$.
\end{itemize} 
Then there is a finite system of polynomials 
$$
\psi_i \in \oo_K[X_1,X_2,\dots, X_{n+1}; Y_1,Y_2,\dots, Y_m; Z_1,Z_2,\dots, Z_t], 
\quad 1 \le i \le k,
$$ 
where we call
\begin{itemize} 
\item 
$X_1,X_2,\dots, X_{n+1}$  the {\bf fundamental} variables,
\item  
$Y_1,Y_2,\dots, Y_m$ the {\bf auxiliary} variables (to take care of items (1)-(6) above), 
\item  
$Z_1, Z_2,\dots,Z_t$ the {\bf congruence} variables,
\end{itemize}  
with the following property.  
Set the  congruence variables $Z_1, Z_2,\dots,Z_t$ to the given elements 
$j_1,j_2,\dots, j_t\in \oo_K$ to obtain a system 
$\Psi_I := \{\psi_{I,1}, \ldots, \psi_{I,k}\}$ defined by
\begin{align*}
\psi_{I,i}(X_1,X_2,\dots, X_{n+1};Y_1,Y_2,\dots, Y_m) 
&:= \psi_i( X_1,X_2,\dots, X_{n+1};Y_1,Y_2,\dots, Y_m; j_1, j_2,\dots, j_t) \\
&\in \oo_K[X_1,X_2,\dots, X_{n+1};Y_1,Y_2,\dots, Y_m].
\end{align*}
If 
$$
(X_1,X_2,\dots, X_{n+1};Y_1,Y_2,\dots, Y_m) \mapsto (\lambda_1,\lambda_2,\dots, \lambda_{n+1}; \mu_1,\mu_2,\dots, \mu_m)  \in \Aff^{n+1+m}(\oo_L)
$$ 
is a common zero (in $\oo_L$) of the system of equations  $\Psi_I$,  
then there exists a rational point $P \in M^r_{I,L}$ such that the 
first $n+1$ entries of that common zero, i.e., the values of the 
{\it fundamental variables}:
$$
(X_1,X_2,\dots, X_{n+1}) \;\mapsto \;(\lambda_1,\lambda_2,\dots, \lambda_{n+1})  \in \Aff^{n+1}(\oo_L)
$$
represent homogenous coordinates for $\iota(P) \in \Aff^{n+1}(\oo_L)$ 
that generate the unit ideal in $\oo_L$:
$$
(\lambda_1,\lambda_2,\dots, \lambda_{n+1})= (\lambda_1(P),\lambda_2(P),\dots, \lambda_{n+1}(P)).
$$
 Moreover, every   $P \in M^r_{I,L}$ is so represented.
 \end{corollary}
 
Now recall Corollary \ref{cor01}: Let $(r,I)$ be an exponent of diophantine stability 
for $G$ relative to $L/K$. Let $Y$ denote the set of $\alpha \in \oo_L$ for 
which there are points $P, Q \in M^r_{ID(\alpha),L}$ with $P \ne {e}$ such that 
$$ 
\partial Q = \alpha\cdot \partial P \; \in\; {\mathcal W}_P\otimes_{\oo_K}\oo_L.
$$ 
Assuming that we have chosen representative homogeneous coefficients for $P$ 
and $Q$ for which there are solutions in equations \eqref{hPQ}, consider an 
equation in the form of \eqref{eqcor1}: 
$$ 
\partial Q = \alpha\cdot \partial P \;\in\; {\mathcal W}_P\otimes_{\oo_K}\oo_L 
$$ 
which can be written (appealing to \eqref{eq3} above):
$$
\delta(P)\cdot \big(\lambda_1(Q), \dots, \lambda_d(Q)\big) 
\equiv \alpha\cdot \delta(Q)\cdot (\lambda_1(P), \dots, \lambda_d(P)) \pmod{z_P^2\cdot (\oo_L)^d}.
$$

In order to rewrite the equivalence above as a polynomial equation we note that 
for any $x \in \oo_L$, we have that $x \in z^2_P\oo_L$ if and only if 
$x=\sum_{i,j}a_{i,j}\lambda_i(P)\lambda_j(P), a_{i,j} \in \oo_L$. As observed 
earlier, any two sets of coordinates corresponding to $P$ will differ by a unit 
and therefore will generate the same ideal. 

We now summarize the discussion above in the following three propositions.
\begin{proposition} 
\label{prop2} 

Let $L/K$ be a number field extension.  Suppose $G$ is a group scheme over $\oo_K$ satisfying Assumption \ref{Ghyp} and 
$G(\oo_L)$  has a point of infinite order. If 
there is an exponent of diophantine stability for  $L/K$, then
there exists an {\em existential definition} $f(t,\bar x) \in \oo_K[t, \bar{x}]$
of $\oo_K$ in $\oo_L$ such that for every $t \in \oo_K$ we have that the equation $f(t,\bar x)=0$ has solutions in $\oo_K$.
\end{proposition}

\begin{proof} 
The assertion follows from Corollaries \ref{cor01} and \ref{exist} combined with 
Lemma \ref{3.7}. 
\end{proof}
In the same fashion we can prove the following proposition.
\begin{proposition}
\label{prop3}
Let $\bL$ be a totally real algebraic extension of $\Q$, possibly of infinite degree. Let $\bK$ be a subfield of $\bL$. Suppose $G$ is a group scheme over $\oo_{\bK}$ satisfying Assumption \ref{Ghyp} and 
$G(\oo_{\bL})$  has a point of infinite order. If 
there is an exponent of diophantine stability for ${\bL}/{\bK}$, then 
\begin{itemize}
\item if ${\bK}$ is a number field we have an {\em existential definition} $f(t,\bar x)$ of $\oo_{\bK}$ over $\oo_{\bL}$  such that for every $t \in \oo_{\bK}$  the equation $f(t,\bar x)=0$ has solutions in $\oo_{\bK}$.
\item if ${\bK}$ is an infinite extension of $\Q$ there exists $D \subset \oo_{\bL}$ such that $D$ has an existential definition $f(t,\bar x)$ over $\oo_{\bL}$, $\N \subset D \subset \oo_{\bK}$ and for any $t \in D$ the equation $f(t, \bar x)=0$ has solutions in $\oo_{\bK}$.

\end{itemize}
\end{proposition}
We now consider the case of quadratic extensions of totally real fields.
\begin{proposition}
\label{prop5}
Let $\bF$ be a quadratic extension of a totally real algebraic extension $\bL$ of $\Q$, possibly of infinite degree over $\Q$.  Let ${\bK} \subset {\bF}$ be a field. Suppose $G$ is a group scheme over $\oo_{\bK}$ satisfying Assumption \ref{Ghyp} and 
$G(\oo_{\bF})$  has a point of infinite order. If 
there is an exponent of diophantine stability for  ${\bF}/{\bK}$, then 
\begin{itemize}
\item if $\bK$ is a number field we have an {\em existential definition} of $\oo_{\bK}$ over $\oo_{\bF}$.  
\item if $\bK$ is an infinite extension of $\Q$ there exists $D \subset \oo_{\bL}$ such that $D$ has an existential definition over $\oo_{\bF}$ and $\N \subset D \subset \oo_{\bK}$.
\end{itemize}
\end{proposition}
\begin{proof}
The proof of the proposition follows from Corollaries \ref{cor:extension2}, \ref{cor01}, \ref{exist} combined with 
Lemma \ref{3.7}. 
\end{proof}

The last proposition in this series does not require an assumption on existence of a group scheme, because we know such a group scheme exists.

\begin{proposition}
\label{prop6}
Let $\bF$ be a quadratic extension of a totally real algebraic extension $\bL$ of $\Q$, possibly of infinite degree over $\Q$.  
There exists a set $D \subset \oo_{\bF}$ such that $D$ has an existential definition over $\oo_{\bF}$ and $\N \subset D \subset \oo_{\bL}$.
\end{proposition}

\begin{proof}
The proof of the proposition follows from Corollaries \ref{cor:algg}, \ref{cor:ext2}, \ref{cor:extension2}, \ref{cor: FtoK} and \ref{exist}.
\end{proof}

Below we state another corollary emphasizing the fact that in the case of totally real 
number fields and quadratic extensions of totally real number fields, our diophantine 
definition of $\oo_K$ over $\oo_M$ does not depend on the degree $[M:\Q]$.

\begin{corollary}
Let ${\bf M}$ be a totally real algebraic extension of $\Q$ or a quadratic extension 
of a totally real algebraic extension of $\Q$.  Let $K \subset {\bf M}$ be a number 
field.  Suppose $G$ is a group scheme over $\oo_K$ satisfying Assumption \ref{Ghyp} 
and $G(\oo_{\bf M})$  has a point of infinite order. Suppose also
there is an exponent of diophantine stability for $G$ relative to ${\bf M}/K$.  
Let $\mathcal M$ be the collection of all number fields $M$ such that 
$K \subset M \subset {\bf M}$.  Then there exists a single diophantine definition of 
$\oo_K$ over $\oo_M$ across all fields $\oo_M \in \mathcal M$.
\end{corollary}

Theorem \ref{mainth1} follows directly from Proposition \ref{prop2},  and Theorem \ref{thm:inf} follows 
directly from Propositions \ref{prop3} and \ref{prop6}.

\section{A simple example}\label{simple} 

\begin{definition} 
A {\df CM-field} ${\bF}$ is a totally complex field of algebraic numbers 
(possibly of infinite degree over $\Q$) possessing an involution $\sigma$ 
with fixed field ${\bF}^+:={\bF}^\sigma$ totally real. 
\end{definition}

\begin{remarks}
\begin{enumerate} 
\item 
This terminology is usually only used for number fields $F/{\Q}$ (i.e., of 
finite degree over $\Q$) such fields being related to complex multiplication on
abelian varieties---hence the ``CM." 
\item 
The involution $\sigma$ referred to in the definition above is unique: there 
is at most one involution of a totally complex field whose fixed field is totally 
real. Equivalently, the field ${\bF}^+$ is the only totally real subfield 
${\bL}$ of ${\bF}$ such that ${\bF}/{\bL}$ is quadratic. We will refer 
to ${\bF}^+$ as {\em the} maximal totally real subfield of ${\bF}$.
\end{enumerate}
\end{remarks}

Let ${\bF}$ be a CM-field and let ${\bL}={\bF}^+$.
Take our group scheme $G = \gm := {\Spec}(\oo_{\bL}[s,s^{-1}])$ to 
be the multiplicative group. We view this as a quasi-projective smooth 
group scheme over $\oo_{\bL}$.
We have that $G(\oo_{\bL}) = \oo_{\bL}^\times \subset G(\oo_{\bF}) = \oo_{\bF}^\times$. 

\begin{lemma}
\label{Dirichlet}
Suppose $I$ is an ideal of $\oo_{\bF}$ divisible by a rational integer $n \ge 3$. 
Then $M^2_{I,{\bF}} \subset \oo_\bL^\times$. That is, the pair $(2, I)$ 
is an exponent of diophantine stability for $G$ relative 
to the field extension ${\bF}/{\bL}$ (see Definition \ref{ids}(2)).
\end{lemma} 

\begin{proof}
Let $\sigma$ denote complex conjugation, the nontrivial automorphism of ${\bF}/{\bL}$. 
Suppose $x \in M_{I,{\bF}}$, i.e., $x \in \oo_{\bF}^\times$ and $x \equiv 1 \pmod{I}$. We have
\begin{equation}
\label{dut}
x^2 = (x x^\sigma)(x/x^\sigma).
\end{equation}
By Dirichlet's unit theorem (or the fact that all absolute values of $x/ x^{\sigma}$ 
are $1$), we have that $x/x^{\sigma}$ is a root of unity. Since $x/x^{\sigma} \equiv 1 \pmod{n}$,
we have $x/x^{\sigma} = 1$. Since $x x^\sigma$ is fixed by $\sigma$, it follows from 
\eqref{dut} that $x^2 \in \oo_{\bL}$.
\end{proof}

This $(r=2,I)$ exponent of diophantine stability for $G$ relative to the extension $\oo_{\bF}/\oo_{\bL}$ 
allows one to prove the following proposition due to Denef for the case when ${\bF}$ 
is a number field.

\begin{proposition}
\label{prop:realsubset}
Let ${\bF}$ be a CM field, and let $K$ be any number field contained in ${\bF}^+$. There 
exists a set ${\mathcal E}_K \subset \oo_{{\bF}^+}$ such that the following 
conditions are satisfied.
\begin{enumerate}
\item ${\mathcal E}_K$ is diophantine over $\oo_{\bF}$,
\item $\oo_K \subset {\mathcal E}_K$.
\end{enumerate}
If ${\bF}^+$ is a number field then we can take ${\mathcal E}_K = \oo_K$.
\end{proposition}

\begin{proof} 
Put ${\bL}:={\bF}^+$.
We have the well-arranged embedding 
$$ 
\iota: \gm \hookrightarrow {\Aff}^1 = \ga= {\Spec}(\oo_{\bL}[ t] )
$$ 
given by $t \mapsto s-1$.
The ideal cutting out the identity section of $\ga$ is $(t)$; the ideal cutting 
out the identity section of $\gm$ is $(s-1)$. So the mapping $G_{[2]} \to {\Aff}^1_{[2]}$ 
is given by the homomorphism (isomorphism, in fact)
$$
\oo_{\bL}[t]/(t^2) \stackrel{\simeq}{\longrightarrow} 
\oo_{\bL}[s,s^{-1}]/((s-1)^2)=\oo_{\bL}[s]/((s-1)^2)
$$
that sends $t\mapsto s-1$.

Let us connect with the notation of Section \ref{VC} and more specifically equations \eqref{p2} and 
\eqref{dual} of that section.

\begin{itemize} 
\item 
As above, $G=\gm= {\Spec}(\oo_{\bL})[ s,s^{-1}] $ and $G_{[2]}={\Spec}(\oo_{\bL}[s]/((s-1)^2)).$
\item 
An $\oo_{\bL}$-rational point $P\in G(\oo_{\bL})$ is given by a homomorphism 
$$
P: \oo_{\bL}[s,s^{-1}] \longrightarrow \oo_{\bL}$$ sending $s$ to a unit 
$u \in \oo_{\bL}^\times \subset \oo_{\bL}$.
\item 
The vanishing ideal $z_P$ is the ideal generated by $u-1$ in $\oo_{\bL}$. 
\item 
The congruence ideal $c_P=z_P^2$ is the ideal generated by $(u-1)^2$. 
\item 
${\mathcal N}_G =(s-1)\cdot \oo_{\bL}[s]/((s-1)^2)$; it is a free $\oo_{\bL}$-module of rank $1$.
\item 
$\partial P: {\mathcal N}_G \to \oo_{\bL}/(c_P)= \oo_{\bL}/((u-1)^2)$ is the 
$\oo_{\bL}$-homomorphism sending $s-1$ to $u-1$. We view $\partial P$ as an element in 
$$
{\mathcal W}_P:={\mathcal N}^*\otimes_\oo(z_P/z_P^2)\subset 
{\mathcal N}^*\otimes_\oo(\oo/z_P^2)
$$ 
as in Lemma \ref{app} above.
\end{itemize} 

Now the ``$Y$" in Corollary \ref{cor01} (combined with Lemma \ref{3.7}) 
gives the desired subset:
$$
{\mathcal E}_K:= Y = \{\alpha \in \oo_L : 
\text{$\alpha$ satisfies \eqref{eqcor1.4} and \eqref{eqcor1.5} below}\}
$$
\begin{equation}
\label{eqcor1.4} 
\text{$\exists u_1,u_2,u_3,u_4 \in \oo_L$ such that $\alpha=1+u_1^2 +\ldots u_4^2$}
\end{equation}\begin{equation}
\label{eqcor1.5}
\exists P, Q \in M^2_{(2\alpha+1)^2,{\bL}} \text{ with }P \ne {e} \text{ and  }
\partial Q = \alpha\cdot \partial P \;\in\; {\mathcal W}_P\otimes_{\oo_K}\oo_{\bL}. 
\end{equation} 
\end{proof}

\part{Diophantine stability in infinite algebraic extensions of $\Q$---results and conjectures}

To date existential undecidability is known for very few rings with fraction fields equal to infinite algebraic extensions of $\Q$. The third author has shown that in any abelian extension of $\Q$ with finitely many ramified primes $\Z$ is existentially definable in infinitely many rings of $\calS$-integers strictly larger than the ring of integers of the field in question. (\cite{Sh17} and \cite{Sh36}). 

All the known results about existential definability of $\Z$ over the ring of integers in infinite extensions require some form of diophantine stability of elliptic curves. The first such results appear in \cite{Sh40} and require diophantine stability of an elliptic curve in a totally real infinite extension of $\Q$. The definability of $\Z$ can then be extended to any quadratic extension of the totally real field under consideration. 

From results of K. Kato \cite{Kato}, K. Ribet \cite{Rib} and D. Rohrlich \cite{Rohr1, Rohr2} we know that in cyclotomic extensions with finitely many ramified primes there exist elliptic curves with groups of points over these fields of positive rank and finitely generated (see \cite{LR}, Theorem 1.2). Thus---as discussed in the introduction (Theorem \ref{KRR} )---$\Z$ is existentially definable in any abelian extension of ${\Q}$ having only finitely many ramified primes.

Below we show that diophantine stability of general abelian varieties, not just elliptic curves, can be used to establish existential definability of $\Z$ over the rings of integers of totally real infinite extensions of $\Q$ and their totally complex quadratic extensions. We also make use of diophantine stability of multiplicative groups as described in Section \ref{simple}. Additionally, we state some conjectures and questions based on conjectures and questions from Section \ref{stable}.

\section{Diophantine stability in abelian extensions}\label{stable}

Our expectation, based on conjectures about elliptic curves over $\Q$ 
(some of which we describe below), 
is that diophantine {\em instability} is rare for abelian varieties in abelian 
extensions unless there is a good reason (analytic or geometric) for it. 
In this section we discuss this expectation and some open questions.

\begin{notation}
If $A/K$ is an abelian variety and $F/K$ is a finite Galois extension, we denote by 
$\mathbf{N}_{F/K} : A(F) \to A(K)$ the map $x \mapsto \sum_{\gamma\in\Gal{F}{K}}\gamma x$.

If $F/K$ is an abelian extension (finite or infinite) 
and $\chi : \Gal{F}{K} \to \C^\times$ is a character of 
finite order, we will say that $\chi$ {\em occurs in} $A(F)$ if $\chi$ occurs in the 
representation of $\Gal{F}{K}$ acting on $A(F) \otimes \C$.

From now on, a {\em character} of a Galois group $G$ means a continuous character, i.e., 
a homomorphism $G \to \C^\times$ of finite order. 
\end{notation}

The following lemma shows that to understand diophantine stability in abelian extensions 
it suffices to understand diophantine stability in cyclic extensions.

\begin{lemma}
\label{ac}
Suppose $\mathbf{L}$ is an abelian extension of a number field $K$, and $A$ is an abelian 
variety defined over $K$. Then the following are equivalent:
\begin{enumerate}
\item
$\rank\, A(F) > \rank\, A(K)$ for some finite extension $F/K$ contained in $\mathbf{L}$,
\item
$\rank\, A(F) > \rank\, A(K)$ for some finite cyclic extension $F/K$ contained in $\mathbf{L}$,
\item
$\{x \in A(F) : \mathbf{N}_{F/K}x = 0\}$ is infinite for some finite cyclic $F/K$ contained in $\mathbf{L}$,
\item
there is a nontrivial character $\chi$ of $\Gal{\mathbf{L}}{K}$ that occurs in $A(\mathbf{L})$.
\end{enumerate}
\end{lemma}

\begin{proof}
The implication $(2) \Rightarrow (1)$ is trivial.

Suppose $F$ is a finite cyclic extension of $K$. 
Let $Z_F := \ker \mathbf{N}_{F/K} \subset A(F)$. 
Then $A(K) \cap Z_F \subset A(K)_\tors$ is finite, and if $x \in A(F)$ then
$$
[F:K]x = \mathbf{N}_{F/K}x + ([F:K] - \mathbf{N}_{F/K})x \;\in A(K) + Z_F.
$$
Thus there is a homomorphism with finite kernel and cokernel
\begin{equation}
\label{aspl}
A(K) \oplus Z_F \longrightarrow A(F),
\end{equation}
so in particular $(3) \Leftrightarrow (2)$.

Now suppose (4) holds, and let $F$ be the cyclic extension of $K$ cut out by $\chi$.
Then $\chi$ occurs in $A(F)$ but not in $A(K)$, so \eqref{aspl} shows that $\chi$ must occur in $Z_F$. 
In particular $Z_F$ is infinite, so $(4) \Rightarrow (3)$.

Finally, if (1) holds, then $A(F)/A(K)$ is infinite, so some character $\chi$ 
of $\Gal{\mathbf{L}}{K}$ occurs in $A(F)/A(K)$, and such a $\chi$ is necessarily nontrivial.
Thus $(1) \Rightarrow (4)$.
\end{proof}

\begin{lemma}
\label{bc}
Suppose $\mathbf{L}$ is an abelian extension of $\Q$, and $A$ is an abelian 
variety defined over $\Q$. Then the following are equivalent:
\begin{enumerate}
\item
$A(\mathbf{L})$ is finitely generated,
\item
the set 
$\{\text{characters $\chi$ of $\Gal{\mathbf{L}}{\Q}$} : \text{$\chi$ occurs in $A(\mathbf{L})$}\}$ is finite.
\end{enumerate}
\end{lemma}

\begin{proof}
The implication $(1) \Rightarrow (2)$ is clear.

Suppose (2) holds. A theorem of Ribet \cite{Rib} shows that $A(\mathbf{L})_\tors$ is finite. 
Let $w$ be the exponent of $A(\mathbf{L})_\tors$. 
Fix a finite abelian extension $K$ of $\Q$, contained in $\mathbf{L}$, such that 
$A(\mathbf{L})_\tors \subset A(K)$ and
all characters $\chi$ of $\Gal{\mathbf{L}}{\Q}$ that occur in $A(\mathbf{L})$ factor through $\Gal{K}{\Q}$.
Then for every finite extension $F$ of $K$ contained in $\mathbf{L}$ we have that
$[A(F):A(K)]$ is finite. 

For every such $F$, define a homomorphism 
$$
\kappa_F : A(F) \longrightarrow H^1(F/K,A(F)_\tors) = \Hom(\Gal{F}{K},A(F)_\tors)
$$
by sending $x \in A(F)$ to the $\Gal{F}{K}$-cocycle $g \mapsto gx-x$. 
Note that $gx-x \in A(F)_\tors$ since $[A(F):A(K)]$ is finite. 
It is easy to see that the kernel of $\kappa_F$ is $A(K)$, and 
hence $A(F)/A(K)$ is killed by $w$. This holds for every field extension $F$
of $K$ in $\mathbf{L}$, so $w$ annihilates $A(\mathbf{L})/A(K)$ as well, i.e., 
$$
A(\mathbf{L}) \subset \{ x \in A(\bar{\Q}) : wx \in A(K)\}.
$$
Since the right-hand side is finitely generated, so is $A(\mathbf{L})$. Thus $(2) \Rightarrow (1)$.
\end{proof}

A lower bound for diophantine stability is given by the 
following theorem from \cite{MR1}.

\begin{theorem}[Theorem 1.2 of \cite{MR1}]
\label{intro:1}
Suppose $A$ is a simple abelian variety over $K$ and all $\bar{K}$-endomorphisms of 
$A$ are defined over $K$. Then there is a set $\calS$ of rational primes with
positive density such that for every $\ell \in \calS$ and every $n \ge 1$, there 
are infinitely many cyclic extensions $L/K$ of degree $\ell^n$ such that 
$A(L) = A(K)$.

If $A$ is an elliptic curve without complex multiplication, then $\calS$ can be 
taken to contain all but finitely many rational primes.
\end{theorem}

\begin{proof}
All but the last sentence is proved in \cite{MR1}.  In the case that $A$ is a non-CM 
elliptic curve, the proof in \cite{MR1} shows that we can take $\calS$ to contain all 
primes $\ell$ such that the $\ell$-adic representation 
$G_K \to \Aut(A[\ell^\infty]) \to \GL_2(\Z_\ell)$ is surjective.  By Serre's theorem 
\cite{serre} this holds for all but finitely many $\ell$.
\end{proof}

The next two conjectures are formulated in \cite{MR2}, where they are motivated 
by the statistical properties of modular symbols. 
They were inspired by earlier conjectures (based on random matrix theory)
given by others; notably David, Fearnley, and Kisilevsky \cite{DFK1, DFK2}.

\begin{conjecture}[Conjecture 10.2 of \cite{MR2}]
\label{mrc1}
Suppose $E$ is an elliptic curve over $\Q$, and $\mathbf{L} \subset \Q^{\ab}$ is a real abelian 
field that contains only finitely many extensions of $\Q$ of degree $2$, $3$, or $5$. Then 
$E(\mathbf{L})$ is finitely generated.
\end{conjecture}

\begin{conjecture}[Conjecture 10.1 of \cite{MR2}, combined with the 
Birch and Swinnerton-Dyer conjecture]
\label{mrc2}
Suppose $E$ is an elliptic curve over $\Q$. Let $X$ denote the set of even characters 
of $\Gal{\Q^\ab}{\Q}$. Then
$$
\{\chi \in X : \text{$\mathrm{order}(\chi) \ge 7$, $\mathrm{order}(\chi) \ne 8, 10$ or $12$, 
and $\chi$ occurs in $E(\Q^\ab)$}\} 
$$
is finite.
\end{conjecture}

We now give some consequences of these conjectures for diophantine stability of 
abelian varieties over $\Q$.

\begin{consequence}
Suppose Conjecture \ref{mrc1} holds, $E$ is an elliptic curve over $\Q$, 
and $\mathbf{L} \subset \Q^{\ab}$ is a real abelian 
field that contains only finitely many extensions of $\Q$ of degree $2$, $3$, or $5$. 
Then there is a finite extension $M/\Q$ such that for every number field $F$ 
satisfying $F \subset \mathbf{L}$ and $F \cap M = \Q$, we have that $F/\Q$ is diophantine 
stable for $E$.
\end{consequence}

\begin{proof}
Take $M$ to be the field generated by the coordinates of points in $E(\mathbf{L})$, so $E(M) = E(\mathbf{L})$. 
Conjecture \ref{mrc1} says that $E(\mathbf{L})$ is finitely generated, so $M/\Q$ is finite. 
If $F \subset \mathbf{L}$ and $F \cap M = \Q$, then 
$$
E(F) = E(F) \cap E(\mathbf{L}) = E(F) \cap E(M) = E(F \cap M) = E(\Q).
$$
\end{proof}

\begin{defprop}
\label{twdef}
Suppose $A$ is an abelian variety defined over a number field $K$, and $F/K$ is a finite cyclic 
extension. Let $G := \Gal{F}{K}$, let $\Q[G]_F$ be the unique irreducible $\Q[G]$-submodule 
of $\Q[G]$ on which $G$ acts faithfully, and let $\Z[G]_F := \Q[G]_F \cap \Z[G]$.

Following \cite[\S5]{MRS} we construct an abelian variety 
$A_F$ over $K$ ({\it the twist of $A$ by $F/K$}) that has these properties: 
\begin{enumerate} 
\item 
$\dim(A_F) = \varphi([F:K]) \dim(A)$, where $\varphi$ denotes the Euler $\varphi$-function.
\item 
The base change of $A_F$ to $F$ is canonically (and $G$-equivariantly) 
isomorphic to $A \otimes \Z[G]_F$ (over $F$).
\item 
Suppose $\mathbf{L}$ is a field containing $K$ and $\mathbf{L} \cap F = K$. Then:
\begin{enumerate}
\setcounter{enumi}{1}
\item
There is a natural inclusion $A_F(\mathbf{L}) \subset A(F\mathbf{L})$ that identifies 
$$
A_F(\mathbf{L}) \cong \{x \in A(F\mathbf{L}) : \text{$\mathbf{N}_{F\mathbf{L}/M\mathbf{L}}x = 0$ 
for every $M$ with $K \subseteq M \subsetneq F$}\}.
$$
\item
Suppose $\chi$ is a character of $\Gal{\mathbf{L}}{K}$. Then 
$\chi$ occurs in $A_F(\mathbf{L})$ if and only if $\chi \rho$ occurs in $A(F\mathbf{L})$ for some 
faithful character $\rho$ of $\Gal{F}{K}$.
\end{enumerate}
\end{enumerate}
\end{defprop}

\begin{proof}
Assertion (1) is \cite[Theorem 2.1(i)]{MRS}, (3.a) is \cite[Theorem 5.8(ii)]{MRS}, 
and (3.b) follows from (3.a).
\end{proof}

\begin{corollary}
\label{twistrem}
Let $F/K$, $A$, and $A_F$ be as in Proposition \ref{twdef}.
If $[F:K]$ is prime, then $\rank\, A_F(K) > 0$ if and only if $\rank\, A(F) > \rank\, A(K)$.
\end{corollary}

\begin{proof}
Taking $\bL=K$ in Proposition \ref{twdef}(3) shows that $\rank\, A_F(K) > 0$ 
if and only if some faithful character $\rho$ of $\Gal{F}{K}$ occurs in $A(F)$.
If $[F:K]$ is prime, then $\rho$ is faithful if and only if it is nontrivial. By 
Lemma \ref{ac}(1,4) a nontrivial $\rho$ occurs in $A(F)$ if an only if 
$\rank\, A(F) > \rank\, A(K)$. This proves the corollary.
\end{proof}

\begin{consequence}
\label{conseq}
Suppose $F/\Q$ is a cyclic extension of prime degree $p \ge 7$, and $E$ is an 
elliptic curve over $\Q$. Let $A := E_F$ be the twist of $E$ as in Definition \ref{twdef}.
Suppose Conjecture \ref{mrc2} holds and $\mathbf{L}$ is a real abelian extension of 
$\Q$ not containing $F$. Then:
\begin{enumerate}
\item
$A(\mathbf{L})$ is finitely generated.
\item
If $\rank\, E(F) > \rank\, E(\Q)$ then $\rank\, A(\Q) > 0$.
\end{enumerate}
\end{consequence}

\begin{proof}
Proposition \ref{twdef}(3b) shows that a character 
$\chi$ of $\Gal{\mathbf{L}}{\Q}$ occurs in $A(\mathbf{L})$ 
if and only if $\chi\rho$ occurs in $E(F\mathbf{L})$ for some nontrivial character 
$\rho$ of $\Gal{F}{\Q}$. Since $\mathbf{L} \cap F = \Q$, we have 
$\Gal{\mathbf{L}F}{\Q} = \Gal{\mathbf{L}}{\Q}\times\Gal{F}{\Q}$, so such a character $\chi\rho$ 
has order divisible by $p \ge 7$. Thus Conjecture \ref{mrc2}
predicts that only finitely many characters occur in $A(\mathbf{L})$. 
Now (1) follows by Lemma \ref{bc}, and (2) follows from (1) by Corollary \ref{twistrem}. 
\end{proof}

The following construction shows that there is a large collection of abelian 
fields $\mathbf{L}$ ``close'' to $\Q^\ab$ to which we can try to apply Consequence \ref{conseq}.

\begin{example}
\label{bigfields}
Fix a prime $p \ge 7$, and another prime $\ell \equiv 1 \pmod{p}$ but 
$\ell \not\equiv 1 \pmod{p^2}$. 
Let $F$ denote the unique cyclic extension of $\Q$ of degree $p$ and conductor $\ell$.  
It follows from class field theory that there are 
infinitely many real abelian extensions $\mathbf{L}/\Q$ such that
\begin{itemize}
\item
$[\Q^{\ab,+}:\mathbf{L}] = p$,
\item
$\mathbf{L} \cap F = \Q$
\end{itemize}
(where $\Q^{\ab,+}$ denotes the real subfield of $\Q^\ab$). 
\end{example}

In order to apply Consequence \ref{conseq} to the arithmetic of the ``big'' fields 
of Example \ref{bigfields} (see Consequences \ref{con:real7} and \ref{con:ab7}), 
we need to have that $\rank\, E(F) > \rank\, E(\Q)$. 
This leads to the following question:

\begin{question}
\label{qc}
Suppose $F/\Q$ is a cyclic extension of prime degree $p$.
Is there an elliptic curve $E$ defined over $\Q$ such that $\rank\, E(F) > \rank\, E(\Q)$?
\end{question}

\begin{remark}
Computer calculations and heuristics similar to \cite{MR2} suggest that when $p = 7$, 
the answer to Question \ref{qc} is ``yes''. When $p > 7$ the answer is less clear, 
but Fearnley and Kisilevsky \cite{DFK2} produce some examples with $p = 7$ and $11$.
Our own calculations, assuming the Birch and Swinnerton-Dyer conjecture, found 
four examples with $p = 13$. For instance, if $E$ is the curve 
labeled 4025.c1 in \cite{LMFDB}, and 
$F$ is the extension of degree 13 in $\Q(\zeta_{53})$, then $L(E/F,s)/L(E/\Q,s)$ 
vanishes at $s=1$. Thus the Birch and Swinnerton-Dyer conjecture predicts that 
$\rank\, E(F) > \rank\, E(\Q)$.
\end{remark}

\begin{consequence}
\label{newcon}
Suppose Conjecture \ref{mrc2} holds, and $p = 7$ or $11$. Then there is an abelian 
variety $A$ over $\Q$ and infinitely many real abelian fields $\bL$ such that 
$[\Q^{\ab,+}:\mathbf{L}] = p$, and $A(\mathbf{L})$ is infinite and finitely generated.

If the Birch and Swinnerton-Dyer conjecture holds for elliptic curves, then 
the same statement holds for $p = 13$.
\end{consequence}

\begin{proof}
If $p = 7$, let $E$ be the elliptic curve $y^2+xy+y=x^3-x^2-6x+5$ and $F$ 
the abelian field of degree $7$ and conductor $29$.
If $p = 11$, let $E$ be the elliptic curve $y^2+xy=x^3+x^2-32x+58$ and $F$ 
the abelian field of degree $11$ and conductor $23$.
Let $\mathbf{L}$ be a real abelian field as in Example \ref{bigfields}, and let
$A := E_F$ be the twist of $E$ as in Definition \ref{twdef}.
Fearnley and Kisilevsky \cite{DFK2} show that $\rank\, E(F) > \rank\, E(\Q)$.
Thus for $p = 7$ or $11$ the desired conclusion follows from Consequence \ref{conseq}.

For $p = 13$ the proof is the same, except that for the elliptic curve 
$E$ labeled 4025.c1 in \cite{LMFDB}, and the cyclic extension $F$ of degree $13$ 
and conductor $53$, we need the Birch and Swinnerton-Dyer conjecture in order to 
conclude that $\rank\, E(F) > \rank\, E(\Q)$.
\end{proof}

The following consequence of Conjecture \ref{mrc2} gives rise to a collection 
of ``big'' abelian fields $\mathbf{L}$ over whose ring of integers 
Hilbert's Tenth Problem has a negative answer. Although these fields 
are not as close to $\Q^\ab$ as those of Example \ref{bigfields}, 
we can produce them without needing to know the answer to Question \ref{qc}.

\begin{consequence}
Suppose Conjecture \ref{mrc2} holds. 
There is a positive integer $n$ and an abelian variety $A/\Q$ such that 
\begin{enumerate}
\item
$\rank\, A(\Q) > 0$,
\item
if $L$ is a finite real abelian extension of $\Q$ and $[L:\Q]$ is relatively prime to 
$n$, then $\rank\, A(L) = \rank\, A(\Q)$,
\end{enumerate}
\end{consequence}

\begin{proof}
Fix an elliptic curve $E$ defined over $\Q$. By Conjecture \ref{mrc2}
there is a finite cyclic extension $F$ of $\Q$ that is maximal in the sense that
\begin{itemize}
\item[(a)]
there is a faithful character $\psi$ of $\Gal{F}{\Q}$ that occurs in $E(F)$,
\item[(b)]
there is no cyclic extension $F'$ of $\Q$ with property (a) that properly contains $F$.
\end{itemize}
Fix such an $F$, and let $n := [F:\Q]$ and $A := E_F$. 
By property (a) and \ref{twdef}(3b), the trivial character occurs in 
$A(\Q)$, so $A(\Q)$ is infinite.

Now fix an abelian extension $L/\Q$ of degree prime to $n$. 
In particular $L \cap F = \Q$.
Suppose $\chi$ is a nontrivial character of $\Gal{L}{\Q}$, and let $L'$ be the 
cyclic extension of $\Q$ cut out by $\chi$. 
Since $[L:\Q]$ is prime to $[F:\Q]$, the compositum $FL'$ is also cyclic over $\Q$. 
By the maximality of $F$ (property (b)) and \ref{twdef}(3.b), we 
conclude that $\chi$ does not occur in $A(L)$.
Lemma \ref{bc} now shows that $\rank\, A(L) = \rank\, A(\Q)$. 
\end{proof}

It is natural to try to generalize Conjectures \ref{mrc1} and \ref{mrc2} by asking 
whether they still hold for abelian varieties over number fields instead of 
elliptic curves over $\Q$.

\begin{question}
\label{e1}
How much diophantine instability can there be?
For example, suppose $K$ is a totally real number field, and 
$A$ is an abelian variety over $K$. Is there a constant $C(A,K)$ such that 
for every finite abelian extension $L/K$, and every character 
$\chi : \Gal{L}{K} \to \C^\times$ of order greater than $C(A,K)$, $\chi$ does {\em not} 
occur in the representation of $\Gal{L}{K}$ on $A(L) \otimes \C$?
If there is such a constant $C(A,K)$, how does it depend on $A$, and on $K$?
\end{question}

\begin{remark}
The reason to restrict to totally real fields in Question \ref{e1} is that 
otherwise the Birch and Swinnerton-Dyer conjecture can be used to force 
diophantine instability. For example, 
suppose $K$ is an imaginary quadratic field, and $E$ is an elliptic curve over $\Q$ 
with the property that every prime where $E$ has bad reduction splits into $2$ distinct 
primes in $K$. Then the theory of Heegner points gives rise to 
arbitrarily large cyclic extensions $L/K$ such that 
$$
\rank\, E(L) > \rank \sum_{K \subseteq F \subsetneq L}E(F).
$$
The fields in question are anticylotomic extensions of $K$, i.e., Galois extensions 
of $\Q$ with $\Gal{K}{\Q}$ acting as $-1$ on $\Gal{L}{K}$. These extensions are ``sparse'' 
in the set of all abelian extensions of $K$.
\end{remark}

\section{Totally real fields}
\label{trf}
The biggest difference between definability problems over finite and infinite extensions of $\Q$ lies in the difficulty of bounding heights of elements in infinite extensions. Recall that for finite extensions, the bound in Section \ref{sec: cong} was generated by using explicitly the degree of the extension over $\Q$. For obvious reasons such a method of producing bounds on the height of elements will not work over a ring of integers of an infinite extension. However, over a totally real field there is a substitute method relying on sums of squares that we used in Section \ref{sec: cong} and other sections of Part 3. So in order to prove existential undecidability over the ring of integers of an infinite totally real extension $\bL$ of $\Q$, all we need is an abelian variety $A$ over $\bL$ with $A(\bL)$ finitely generated and of positive rank.

\begin{corollary}
\label{cor:infty}
Let ${\mathbf{L}}$ be a totally real infinite extension of $\Q$ and let $A$ be an abelian variety such that $A(\mathbf{L})$ is infinite and finitely generated. Then $\Z$ has a diophantine definition over $\oo_{\mathbf{L}}$.
\end{corollary}

\begin{proof}
Let $K$ be the field generated over $\Q$ by points in $A(\mathbf{L})$. 
Then $K$ is a totally real number field and $A(K) = A(\bL)$. 
By Proposition \ref{prop3}, we have that $\oo_K$ is diophantine 
over $\oo_{\bL}$. As was described in the introduction, J. Denef showed 
that the ring of integers of any totally real number field has a diophantine 
definition of $\Z$. Therefore from Lemma \ref{3.5} we can deduce that $\Z$ is 
diophantine over $\oo_{\bL}$. 
\end{proof}

\begin{consequence}
Suppose Conjecture \ref{mrc1} holds, and $\mathbf{L} \subset \Q^{\ab,+}$ is a real abelian 
field that contains only finitely many extensions of $\Q$ of degree $2$, $3$, or $5$. Then $\Z$ has a diophantine definition over $\oo_{\mathbf{L}}$ and Hilbert's Tenth Problem for  $\oo_{\mathbf{L}}$ has a negative solution.
\end{consequence}

\begin{proof}
By Conjecture \ref{mrc1}, we can find an elliptic curve $E$ such that $E(\mathbf{L})$ is infinite and finitely generated. Thus the assertion of the consequence holds by Corollary \ref{cor:infty}.
\end{proof}

From Consequence \ref{newcon} combined with Corollary \ref{cor:infty}, we also obtain the following consequence.

\begin{consequence}
\label{con:real7}
Suppose Conjecture \ref{mrc2} holds. 
If $p =7$ or $11$, then there exists a totally real abelian extension $\mathbf{L}$ such that $[\Q^{\ab,+}:\mathbf{L}]=p$ and $\Z$ has a diophantine definition over $\oo_{\mathbf{L}}$.

If the Birch and Swinnerton-Dyer conjecture holds for elliptic curves, then the 
same is true with $p = 13$.
\end{consequence}

\section{Quadratic extensions of totally real fields again}
\label{qtrf}
\subsection*{A reduction to the maximal totally real subfield}
As has been noted above, totally real fields are special in the sense that sums of squares allow us to impose bounds on heights of variables using existential language of the rings only. Much of this definability ``advantage'' is inherited by quadratic extensions of totally real fields; in other words one can reduce a definability problem over a quadratic extension of a totally real field to a definability problem over this totally real field.
From Proposition \ref{prop5} we have the following corollary.  

\begin{corollary}
Let $\bF$ be a quadratic extension  of a totally real field. 
If there exists an abelian variety $A$ over $\bF$ such that $A(\bF)$ is infinite 
and finitely generated, then $\Z$ is existentially definable over $\oo_{\bF}$, 
and Hilbert's Tenth Problem is undecidable over $\oo_{\bF}$.
\end{corollary}

\begin{proof}
Let $\bF = \bL(\gamma)$ where $\bL$ is totally real and $\gamma^2 \in \bL$.
Let $F$ be the subfield of $\bF$ generated over $\Q$ by the points in $A(\bF)$.  
Since $A(\bF)$ is finitely generated, $F$ is a number field, so we have $F = \Q(\delta)$ 
where $\delta \in \bF$, i.e., $\delta = a + b\gamma$ with $a,b \in \bL$.  
Let $L := F(a,b,\gamma^2) \cap \bL$ and $K := L(\delta)$.  
Then $L$ is totally real, and $a, b, \gamma^2 \in L$, so $[K:L] = 2$.  
Thus $K$ is a quadratic extension of a totally real number field, and $F \subset K$ 
so $A(K) = A(\bF)$.

By a result of Denef (\cite{Den1}) and a result Denef and Lipshitz (\cite{Den2}) 
we have that $\Z$ has a diophantine definition over $\oo_K$.  By Proposition \ref{prop5} 
we have that $\oo_K$ has a diophantine definition over $\oo_\bF$.  Finally by 
Lemma \ref{3.5} (Transitivity of diophantine definitions), we can now conclude 
that $\Z$ has a diophantine definition over $\oo_{\bF}$ and the assertion of the 
corollary follows.
\end{proof}

Then following corollary provides a slightly different way of establishing diophantine undecidability of quadratic extensions of totally real fields.

\begin{corollary}
\label{cor+}
Let $\mathbf{F}$ be a quadratic extension of a totally real field $\bL$.  
If there exists an abelian variety $A$ over $\bL$ such that $A(\bL)$ is infinite 
and finitely generated, then $\Z$ is existentially 
definable over $\oo_{\bF}$, and Hilbert's Tenth Problem is undecidable over $\oo_{\bF}$.
\end{corollary}

\begin{proof}
Let $K$ be the number field generated by the points in $A(\bL)$, so $A(K) = A(\bL)$.
By Proposition \ref{prop3}, there exists a diophantine definition $f(t,\bar x)$ of $\oo_K$ over $\oo_{\bL}$ such that for all $t \in \oo_K$ the equation $f(t,x_1,\ldots, x_r)=0$ has solutions in $\oo_K$. By Proposition \ref{prop6}, we have that there exists $D \subset \oo_{\bF}$ such that $D$ is diophantine over $\oo_{\bF}$ and $\N \subset D \subset \oo_{\bL}$. 

Let $\gamma \in \oo_K$ generate $K$ over $\Q$ and define $\hat D \subset  \oo_{\bF}$ in the following manner.
\[
\hat D=\{\sum_{i=0}^{[K:\Q]-1}a_i\gamma^i| \pm a_i \in D\}.
\]
Observe that since $D$ is diophantine of over $\oo_{\bF}$, we have that $\hat D$ is diophantine over $\oo_{\bF}$.  Further, $\oo_K \subset \hat D \subset \oo_{\bF}$.  Let $g(x,\bar y)$ be a diophantine definition of $\hat D$ over $\oo_{\bF}$.  Now consider the following system of equations:
\begin{equation}
\label{sys:hatD}
\left \{
\begin{array}{c}
g(x_1,\bar y_1)=0,\\
\ldots\\
g(x_r, \bar y_r)=0,\\
g(t,z_1,\ldots,z_r),\\
f(t, x_1,\ldots,x_r)=0
\end{array}
\right .
\end{equation}

Suppose \eqref{sys:hatD} has solutions in $\oo_{\bF}$.  Then by assumption on $g$ being the diophantine definition of $\hat D$, we have that $t, x_1, \ldots,x_r \in \hat D\subset \oo_{\bL}$.  Since $f$ is a diophantine definition of $\oo_K$ over $\oo_{\bL}$, we conclude that $t \in \oo_K$.  
Conversely, suppose $t \in \oo_K$, then there exist $x_1,\ldots,x_r \in \oo_K \subset \hat D$ such that $f(t, x_1,\ldots,x_r)=0$.  Since $t, x_1,\ldots, x_r \in \oo_K \subset \hat D$, there exist $\bar y_1,\ldots, \bar y_r, \bar z$ with all components in $\oo_{\bF}$ such that all $g$ -equations are satisfied.  Thus \eqref{sys:hatD} is a diophantine definition of $\oo_K$ over $\oo_{\bF}$.

  From \cite{Den1} we have that $\oo_K$ has a diophantine definition of $\Z$.  Thus applying Lemma \ref{3.5} to the tower $\Z\subset \oo_K \subset \oo_{\bF}$, we have that $\Z$ has a diophantine definition over $\oo_{\bF}$. 
\end{proof}

Combining Corollary \ref{cor+} with Consequence \ref{newcon} we get another 
consequence of Conjecture \ref{mrc2}.

\begin{consequence}
\label{con:ab7}
Suppose Conjecture \ref{mrc2} holds. 
If $p =7$ or $11$, then there exists an abelian extension $\mathbf{L}$ such that 
$[\Q^{\ab}:\mathbf{L}]=p$ and $\Z$ has a diophantine definition over $\oo_{\mathbf{L}}$.

If the Birch and Swinnerton-Dyer conjecture holds for elliptic curves, then the 
same is true with $p = 13$.
\end{consequence}


\part*{Appendix A. \quad A geometric formulation of diophantine stability}

\section{The same structures described in a different vocabulary}

If $K$ is a number field let $\oo_K$ denote its ring of integers. 
\begin{definition}
\label{system}
(Compare with Definition \ref{def0}.)
Let $L/K$ be an extension of number fields.
Let 
$$
\fbox{${{{\mathcal F}}}: \ f_i({\st}; x_1,x_2,x_3,\dots x_n)$}
$$
be a system of $m$ polynomials ($i=1,2,\dots,m$) with coefficients in $\oo_K$. 
We've singled out the first variable ${\st}$, which will play a special role.
Say that ${\mathcal F}$ is {\df {diophantine stable at ${\st}$}} for $L/K$ if 
all the simultaneous solutions 
$$
f_i({\sf a}; a_1,a_2,\dots, a_n) = 0
$$
for ${\sf a}, a_1, a_2,\dots, a_n \in \oo_L$ and $ i=1,2,\dots, m$ have the 
property that {\em the ``singled out''} entry ${\st}={\sf a}$ lies in $\oo_K$.
\end{definition}

\begin{example}
Let ${\mathcal F}$ be the equation over $\oo_K$ that says that ${\st}$ is a unit:
$$ 
\fbox{$f({\st}; x_1) := 1-{\st}\cdot x_1 =0$.}
$$
So ${\mathcal F}$ is diophantine stable at ${\st}$ for any $L/K$ where $\oo_L$ 
and $\oo_K$ have the same unit group.
\end{example}

Any system of equations ${\mathcal F}$ over $\oo_K$ (as in Definition \ref{system} 
above) determines a a finitely presented affine $\oo_K$-scheme 
$$
V=V_{\mathcal F}:= {\Spec}(\oo_K[{\sf t}; x_1,x_2,\dots,x_n]/(f_1,f_2,\dots,f_m))
$$
Let ${\Aff}^1= {\Spec}(\oo_K[t])$ be $1$-dimensional affine space, viewed as 
(an affine) scheme over $\oo_K$. 
The homomorphism 
$\oo_K[{\st}] \rightarrow \big\{\oo_K[{\st}; x_1,x_2,\dots,x_n]/(f_1,f_2,\dots,f_m)\big\}$ 
induced by sending ${\st} \mapsto {\st}$ can be viewed as an $\oo_K$-morphism:
$$
V \xrightarrow{~\st~} {\Aff}^1
$$
which in turn induces a map on $\oo_L$-valued points
\begin{equation}
\label{Fdst}
V(\oo_L) \xrightarrow{~\st~} \oo_L.
\end{equation}
The {\em diophantine stable at $\st$} property of ${\mathcal F}$ relative to 
$L/K$ is equivalent to the property that the image of $V(\oo_L)$ 
under \eqref{Fdst} is contained in the subset $\oo_K$ of $\oo_L$. 

If we denote the image of \eqref{Fdst} in $\oo_L$ by $E$, then 
$$
V(\oo_L) \xrightarrow{\;\text{onto}\;} E \subset \oo_L= {\Aff}^1(\oo_L)
$$
shows that $\calF$ (or equivalently, the pair $(V,\st)$) is a 
{\df diophantine definition of $E$ over $\oo_L$} (see Definition \ref{defdd}). 
We are especially interested in the case where $E = \oo_K \subset \oo_L$. 

\begin{prop} 
If $L/K$ is an extension of number fields, $\calF$ is a system of 
polynomials as in Definition \ref{system}, and the image $E$ of the 
map \eqref{Fdst} on $\oo_L$-valued points of $V_{\calF}$ satisfies 
$\N \subset E \subset \oo_K$, then there is a system $\calF'$ such
that the corresponding pair $(V_{\calF'},\st)$ is a diophantine 
definition of $\oo_K$ over $\oo_L$.
\end{prop}

\begin{proof} 
This is Lemma \ref{3.7} above.
\end{proof}

A diophantine definition $(V,\st)$ of $\oo_K$ over $\oo_L$ can be used to 
transport any algorithm that determines whether a system of polynomials 
with coefficients in $\oo_L$ has a solution over $\oo_L$ to a similar algorithm
 for systems of polynomials with coefficients in $\oo_K$, as follows. 

\begin{construction}
\label{con5.2} 
Suppose we are given a diophantine definition $(V,{\st})$ of $\oo_K$ over $\oo_L$. 
For every finitely presented affine scheme ${\mathcal B}$ over $S={\Spec}(\oo_K)$ 
we can construct an $S$-scheme ${\mathcal V}={\mathcal V}_{\mathcal B}$ 
with a surjective $S$-morphism $\tau : {\mathcal V} \twoheadrightarrow {\mathcal B}$: 
$$
\xymatrix{{\mathcal V}\ar@{->>}[r]^\tau\ar[dr] & {\mathcal B}\ar[d]\\
\ & S}
$$
with the property that the image of the set of $\oo_L$-valued points of 
${\mathcal V}$ under $\tau$ is equal to the set of $\oo_K$-valued points of ${\mathcal B}$:
$$
\xymatrix{\ & {\mathcal V}(\oo_L)\ar[d]^\tau\ar@{->>}[dl]_{\rm onto} \\ 
{\mathcal B}(\oo_K)\ar[r] & {\mathcal B}(\oo_L)}
$$
\end{construction}

\begin{proof} 
Let $(V,{\st})$ be the diophantine definition, and 
$$
{\mathcal G}:\; g_i(z_1,z_2,\dots,z_\nu)\in \oo_K[z_1,z_2,\dots, z_\nu] 
\quad \text{for}\; i=1,2,\dots,\mu
$$ 
the presentation of the affine scheme ${\mathcal B}$. We can view this 
presentation as giving us a closed embedding 
$$ 
{\Spec}(\oo_K[z_1,z_2,\dots, z_\nu]/(g_1,g_2\dots, g_\mu)) 
= {\mathcal B}\stackrel{j}{ \hookrightarrow} {\Aff}^\nu 
= {\Spec}(\oo_K[z_1,z_2,\dots, z_\nu]).
$$

Let $V^{\{\nu\}}:=V\times_{\oo_K} V\times_{\oo_K}\dots V$ be the $\nu$-fold 
power of $V$ (fiber-product over $S={\Spec}(\oo_K)$) and form the cartesian diagram: 
\begin{equation}
\label{tau2}
\raisebox{22pt}{\xymatrix{{\mathcal V}_{\mathcal B} = {\mathcal V}\ar@{^(->}[r]\ar[d]^\tau 
& V^{\{\nu\}}\ar[d]_\phi^{({\sf t},{\sf t},\dots,{\sf t})}\\
{\mathcal B}\ar@{^(->}[r] & {\Aff}^\nu.}}
\end{equation}

Since the map $\phi$ of \eqref{tau2} is a surjective morphism of schemes, 
so is the projection $ \tau:{\mathcal V} \to {\mathcal B} $. Since $(V,{\st})$ is 
a diophantine definition, the mapping $\phi$ is a surjection of 
$V^{\{\nu\}}(\oo_L)$, the set of $\oo_L$-valued points of $V^{\{\nu\}}$, 
onto ${\Aff}^\nu(\oo_K)$. If $v$ is an $\oo_L$-valued point of ${\mathcal V}$ 
then, by commutativity of \eqref{tau2}, $\tau(v) $ is an $\oo_K$-valued point 
of ${\mathcal B}$; and, by cartesian-ness of \eqref{tau2}, any $\oo_K$-valued 
point of ${\mathcal B}$, viewed in ${\Aff}^\nu(\oo_K)$ lifts to an 
$\oo_L$-valued point of $V^{\{\nu\}}$.
\end{proof}

Therefore
\begin{corollary}
\label{simsol} 
Suppose $(V,{\st})$ is a diophantine definition 
of $\oo_K$ over $\oo_L$. Then the following are equivalent:
\begin{enumerate}
\item 
The $\oo_K$-scheme ${\mathcal B}$ has an $\oo_K$-rational `point' 
(meaning: a ${\Spec}(\oo_K)$- section).
\item 
The system of equations
${\mathcal G}: g_i(z_1,z_2,\dots,z_\nu)\in \oo_K[z_1,z_2,\dots, z_\nu]$ 
for $i=1,2,\dots,\mu$ has a simultaneous solution in $\oo_K$.
\item 
The finitely presented $\oo_L$-scheme ${\mathcal V}_{\mathcal B}$ has an 
$\oo_L$-rational `point' (meaning: a ${\Spec}(\oo_L)$-section).
\item 
The system of equations over $\oo_L$ finitely presenting the 
$\oo_L$-scheme ${\mathcal V}_{\mathcal B}$ has a simultaneous solution in $\oo_L$.
\end{enumerate}
\end{corollary}

\begin{remark}
Suppose ${\mathcal G}$ is a finite system of polynomial equations over $\oo_K$ 
defining a scheme ${\mathcal B}$ as in Construction \ref{con5.2}, and we are 
given a diophantine definition $(V,{\st})$ of $\oo_K$ over $\oo_L$. 
If we have a finite algorithm to determine whether or not a finite system of 
polynomial equations over $\oo_L$ has a simultaneous solution over $\oo_L$, 
then---by Corollary \ref{simsol}---applying this algorithm to the system of 
equations over $\oo_L$ that finitely present the $\oo_L$-scheme 
${\mathcal V}_{\mathcal B}$ will tell us whether or not ${\mathcal G}$ has a 
simultaneous solution over $\oo_K$. In particular, a negative answer for 
$\oo_K$ to the question posed by Hilbert's Tenth Problem implies a similar 
negative answer for $\oo_L$.
\end{remark}

\begin{remark}
What can be said about the category comprising the various diophantine 
definitions of rings of integers related to a given $L/K$? E.g., beyond 
the fact that: 
\begin{itemize} 
\item 
The diophantine definitions, $(V, {\st})$, of $\oo_K$ over $\oo_L$ are 
closed under fiber product over ${\Aff}^1$. 
\item 
Any $(V, {\st})$ sandwiched between two diophantine definitions of $\oo_K$ 
in $\oo_L$ is again one:
$$
\xymatrix{V_1\ar[r]\ar[rd] & V\ar[r]\ar[d]^{\mathbf {t} } & V_2\ar[dl] \\
\ & {\Aff}^1 & \ }
$$
\end{itemize} 
\end{remark}

\begin{question} Given $L/K$ what is the smallest Krull dimension of a 
diophantine definition $(V, {\st})$ of $\oo_K$ over $\oo_L$? For example, 
what is the smallest Krull dimension of a diophantine definition of 
${\Z}$ over ${\Z}[i]$?
\end{question}

A related question concerns the smallest number of variables one needs to 
define $\oo_K$ over $\oo_L$. The smallest number of variables question 
has a long history. In its first version, the question concerned the 
smallest number of variables necessary to define a non-recursive 
c.e. subset of natural numbers or integers. 
Yu. Matiyasevich, J. Robinson and J. Jones were the first people considering 
this problem. Later on they were joined by Zhi-Wei Sun, among others. 
His recent paper \cite{Su2021} contains the most recent survey of the 
results in the area. H. Pasten in \cite{Pas22} and independently A. Fehm, P. Dittman and N. Daans in \cite{ADD21} considered the smallest 
number of variables question in the context of diophantine definitions over 
rings and fields. They called this number the {\em diophantine rank} of a set.


\end{document}